\documentclass{article}

\usepackage{amsmath,amssymb,latexsym,amsthm,mathrsfs, bbm, bm}
\usepackage{comment,appendix}
\usepackage{color}
\usepackage{tikz-cd}
\newtheorem{theo}{Theorem}[section]

\newtheorem{lem}[theo]{Lemma}
\newtheorem{cor}[theo]{Corollary}

\theoremstyle{definition}
\newtheorem{defin}[theo]{Definition}
\newtheorem{exa}{Example}[section]

\theoremstyle{remark}
\newtheorem{rem}[theo]{Remark}

\title{Examples of stable-like random walks on groups of polynomial growth}

\author{Laurent Saloff-Coste and Ruoqi Zhang \thanks{both authors' research was partially supported by NSF grants DMS-2054593 and DMS-234386.}}

\begin{document}

\maketitle

\begin{abstract}
We consider several  families of long jump random walks on groups of polynomial volume growth which are naturally expected to have a stable-like behavior. We then prove optimal pseudo-Poincar\'e inequalities for these walks. These pseudo-Poincar\'e inequalities allow us to show that the random walks in questions indeed have a stable-like behavior and to obtain detailed estimates.
\end{abstract}

\section{Introduction}
This work is concerned with the exploration of stable-like random walks on certain non-commutative groups, in particular, nilpotent groups and groups of polynomial volume growth. The word ``exploration'' is used here because there is no standard definition of what ``stable-like'' means and, consequently, we will focus in providing classes of  natural examples in search of a more formal theory. Taking a  wider viewpoint, we focus on the question of understanding what features of a particular driving probability measure  determine the basic behavior of the associated random walk. Borrowing from classical probability theory, we start with the following question: what features of  a driving probability measure produce a stable-like behavior? In the remaining part of this introduction, we explain why this question is a natural intermediate step. For a general introduction to random walks on groups, see \cite{Lalley,SCnotices,Woess} among other references.

\subsection{The classical case in a nutshell} Even in the context of $\mathbb Z^d$ and $\mathbb R^d$, there is no such things as ``stable'' discrete random walks although specialists would certainly be able to suggest some possible definitions. See, e.g.,  \cite{MS1} for background information,

A self-similar stochastic process (on $\mathbb R$) is a continuous time stochastic process $(X_t)_{t\ge 0}$ having the property  that $$(X_{st})_{t\ge 0} \mbox{ is equal in distribution to }(s^{1/\alpha} X_t)_{t\ge 0},$$ for some $\alpha>0$. 

In the context of Levy processes (processes with independent time homogeneous increments),  a process $(X_t)_{t> 0}$ (started at $0$) is stable if the characteristic function $\hat{\mu}_1$ of $X_1$ has the property that
for all $a>0$ there are $b>0$ and $c\in \mathbb R$ such that
\begin{equation}\forall\, y,\;\;(\hat{\mu}(y))^a=\hat{\mu}(by)e^{icy}.\label{stable}\end{equation}
When this equality holds, it must be that $b(a)=a^{1/\alpha}$ with $\alpha\in (0,2]$. The parameter $\alpha$ is called the index of stability and the process is called $\alpha$-stable. One also says that $X_1$ is an $\alpha$-stable random variable or that its law, $\mu_1$,  is $\alpha$-stable. 

Self-similar Levy processes with $X_0=0$  are exactly the stable processes. More precisely, they are Brownian motion with drift if $\alpha=2$, and 
stable jump processes if $\alpha\in (0,2)$. In the present work, we are mostly interested in symmetric processes (on $\mathbb R$ or $\mathbb R^d$, this means that $X_t$ is equal in law to $-X_t$) and, in this case, the law $\mu_t$ of $X_t$ has characteristic function $\hat{\mu}_t(y)=e^{\kappa |y|^\alpha}$. Unfortunately, in general, the density of $\mu_t$ does not have a simple explicit form, except in the cases when $\alpha=2$
(the Gaussian case) where, assuming as we may that $\kappa=1$, $$\frac{d\mu_t}{dx}(x)=(4\pi t)^{-1/2}e^{-|x|^2/4 t},$$ 
and also in the case $\alpha=1$ (Cauchy distribution) where  $$\frac{d\mu_t}{dx}(x)=\frac{1}{\pi }\frac{ t}{t^2+x^2}.$$
Even though there is no explicit formula  when $\alpha$ is in $ (0,2)\setminus \{1\}$, one can show that 
$$  \frac{c_\alpha t }{( t+|x|^\alpha)^{1+1/\alpha} }\le \frac{d\mu_t}{dx}(x) \le   \frac{ C_\alpha t }{( t+|x|^\alpha)^{1+1/\alpha}} .$$ 

Note that self-similarity appears naturally when considering the possible limits $Z$ of 
sums of independent equidistributed (centered) random variables $(X_i)_1^\infty$ after rescaling, 
$$n^{-1/\alpha}\left(\sum_1^nX_i\right) \Longrightarrow Z,$$
because these eventual limits must be fixed points for the rescaling method used to obtain them. It follows that, in the case of symmetric random variables, the possible non-zero limits are exactly the symmetric stable laws (the rescaling index $\alpha$ must be in $(0,2]$ and is also the stability index of the limit). The cases $\alpha=2$ and $\alpha\in (0,2)$ correspond to drastically different behaviors of the corresponding processes:  path continuity in the case $\alpha=2$ and jumps in the case of $\alpha\in (0,2)$. For clarity, let us explicitly agree to separate the case $\alpha=2$ (the Brownian case) from the case $\alpha\in (0,2)$. From now on, in this work, ``stable'' and ``stable-like'' both will refer to the case $\alpha\in (0,2)$.

Now, consider a symmetric probability measure $\nu$ on the integer group  $\mathbb Z$ with generating support.  What features of $\nu$ would make it ``stable-like''?  Classical probability theory (the central limit theorem), tells us that if $\nu$ is  finitely supported then it is ``Brownian-like,'' and that remains true as long as $\nu$ has a finite second moment
$\sum_{z\in \mathbb Z}|z|^2\nu(z)<+\infty$.   For $m\ge 1$, set
\begin{equation}\label{defKG}
\mathcal K(m)=m^{-2}\sum_{|z|\le m}|z|^2\nu(z),\;\; \mathcal G(m)=\sum_{|z|\ge m}\nu(z).\end{equation}
In 1984, Griffin, Jain and Pruitt \cite{GJP} proved that  if  $\nu$ is symmetric aperiodic and 
\begin{equation}\label{GJP}\liminf_{m\to \infty} \mathcal K(m)/\mathcal G(m)>0\end{equation}
then
\begin{equation}\nu^{(n)}(0)\asymp 1/a_n \label{est}\end{equation}
where $a_n$ is defined by $\mathcal Q(a_n)=1/n$ with $\mathcal Q=\mathcal K+\mathcal G$.  Here $f\asymp g$ means that there are constant $c_1,c_2$ such that $c_1f\le g\le c_2f$ on the adequate domain of definition of $f$ and $g$.  We use $\succcurlyeq$ and $\preccurlyeq$ for the associated order relations.

This result suggests that it may be reasonable to call a symmetric probability measure $\nu$ on $\mathbb Z$  ``$\alpha$-stable-like'' when (\ref{GJP}) holds and $a_n\asymp n^{1/\alpha}$ for some $\alpha\in (0,2)$.  A typical example is
$$\nu(z)=\frac{c_\alpha}{(1+|z|)^{1+\alpha}}, \;\;z\in \mathbb Z.$$
Note that, although the proofs in \cite{GJP} use Fourier transform/characteristic function techniques, the statement above (i.e., (\ref{GJP}-(\ref{est}))does not refer to characteristic functions. This is essential for our purpose.  Griffin \cite{Griffin1986} gives a remarkable extension of these ideas to $\mathbb R^d$-valued random variables which shows that the situation is  more complicated in this case, as one is lead to realize that much richer rescaling methods must be considered in higher dimension.  Namely, the multiplication by $b>0$
in (\ref{stable}) should now be replaced by the action of an element $B$ of $GL_n$ (an invertible linear map) and $cy$ should be replaced by scalar product $c\cdot y$ of a vector $c$ with the vector $y\in \mathbb R^d$.  Then (\ref{stable}) is equivalent to say that the probability measure $\mu$ on $\mathbb R^d$ is embedded in a convolution semigroup of probability measure $\mu^t$ (i.e,  $\mu=\mu^1$) and $\mu_t= B_t(\mu)*\delta_{c_t}$ for all $t>0$. Moreover, $B_t=t^E=\exp(E\log t)$ for some endomorphism $E$. For instance, $E$ could be given in a linear basis $(e_1,\dots, e_d)$ by $Ee_i=(1/\alpha_i)e_i$, $1\le i\le d$. In this case, the $\alpha_i$ must belong to $(0,2]$ and they can be interpreted as directional stability indexes for the measure $\mu$. Several books have been written on the subject of operator-stable distributions but Griffin's article \cite{Griffin1986} is a rare instance of a work that consider the following problem: Given a probability measure $\mu$ which is suspected to be operator-stable (or operator-stable-like), how does one go about finding the matrix $E$? (i.e., loosely speaking, the indices of stability in different directions; in reality, in full generality, operator-stability may also involves rotations). What features of $\mu$ should we explore in order to guess at $E$?
An interesting recent work in this direction is \cite{MM}. In \cite{Griffin1986} and in \cite{MM}, the problem is presented in a less naive form as finding the correct matrix normalization for partial sums of iid random variables, or for Levy processes.

\subsection{Random walks on groups}
  Let us switch gears and jump from the background considerations above to the problem we want to consider here, random walks with stable-like behavior on finitely generated groups. Let $G$ be a finitely generated group with identity element $e$.  For any probability measure $\mu$ on $G$, the random walk driven by $\mu$ is the $G$-valued left-invariant Markov process $(X_i)_{i\ge 0}$ obtained by considering a starting point $X_0=x\in G$, a sequence of independent $G$-valued random variables $(\xi_i)_1^\infty$ with common distribution $\mu$, and setting
$$X_0=x, \;\;X_j=x\xi_1\dots \xi_j, j\ge 1.$$
In other words, $X_0=x$ and $X_{i+1}=X_i\xi_{i+1}$, $i\ge 0$.  
This is, in an  obvious way, a generalization of partial sums of iid sequences of vector-valued random variables. In this work we focus mostly on the case when $\mu$ as the symmetry property $\mu(x)=\mu(x^{-1})$. 

What should we ask first if we want to understand the behavior of such a process?  One reasonable answer to this question is to focus on ``the probability of return.'' In technical terms, this is the function $\Phi_\mu: \{0,1,2,3,\dots\}\to [0,1]$ given by
$$\Phi_\mu(n)=\mathbf P_e(X_{2n}=e).$$
The restriction to even times is justified by the fact that, under the symmetry hypothesis on $\mu$, $n\mapsto \Phi_\mu$
is a non-increasing function.  Assuming that the support of $\mu$ generates an infinite subgroup of $G$, $\lim_{n\to +\infty}\Phi_\mu(n)=0$.

Now, something remarkable happens \cite{PSCstab}.  Equip the finitely generated group $G$ with a symmetric finite generating set $\Sigma$
and the associated word-length $|g|$, the minimal length $k$ of a sequence $\sigma_i\in \Sigma$, $1\le i\le k$, such that $g=\sigma_1\sigma_2\dots \sigma_k$. We say that a probability measure $\mu$ has finite second moment if $\sum_{g\in G}|g|^2\mu(g)<+\infty$.
Let us introduce the equivalence relation $\simeq$ between positive monotone functions. Write  $f\simeq g$ when there exist  positive constants $c_i$, $1\le i\le 4$, for which $f(t)\le c_1 g(c_2t)$ and $g(t)\le c_3f(c_4t)$
on the (common) domain of definition of $f,g$. We use $\gtrsim$ and $\lesssim$ for the associated order relations. Please, note the differences between the equivalence relations $\asymp$ (constants outside the functions) and $\simeq$ (constants both inside and outside the functions).
\begin{theo} \label{PSC}There exists a function $\Phi_G:  \{0,1,2,3,\dots\}\to [0,1]$ such that, for any symmetric probability measure $\mu$ with generating support and finite second moment,  $$\Phi_\mu\simeq \Phi_G.$$
That is, 
there exists  positive constants $c_i$, $1\le i\le 4$, depending on $\mu$, for which$$\forall\, n,\;\;   \Phi_\mu(n)\le c_1\Phi_G(c_2n) \mbox{ and } \Phi_G(n)\le c_3\Phi_\mu(c_4n).$$
\end{theo}
In the classical setting of $\mathbb Z^d$, this theorem deals with random walks which, after proper time-space scaling, tend to Brownian motion.  If we want to focus on stable-like behavior (with index in $(0,2)$), we need to consider measures $\mu$ which do not have finite second moment.  
A natural weaker condition (weaker when $\alpha\in (0,2]$) to consider is the following: 
\begin{defin}
\label{def:weak}
A measure $\mu$ is said to have finite weak $\alpha$-moment ($\alpha>0$) if 
$$\sup _{t>0}\{t^\alpha \mu\{g\in G: |g|\ge t\})\}<+\infty.$$
\end{defin}
This condition means that 
$\mu(\{g:|g|\ge r\})\le C r^{-\alpha}, \;r\ge 1$. In \cite{BSCsub,BSClow, SCZlow}, lower bounds on $\Phi_\mu$ are given in terms of $\Phi_G$ for such measures. In the classical setting of $\mathbb R^d$ and $\mathbb Z^d$, $\alpha$-stable like measures  have finite weak-$\alpha$-moment and no higher finite weak moments (see the definition of the function $\mathcal G$ at (\ref{defKG})).
Here we are interested in the case when $\alpha\in (0,2)$. The results in \cite{BSCsub,BSClow, SCZlow} provide good lower bounds for $\mu^{(2n)}(e)$ when $\mu$ is symmetric with finite $\alpha$-moment, $\alpha\in (0,2)$, and the question arises  to find conditions on $\mu$ that provide matching upper-bounds.
More generally, given a measure $\mu$ that is spread-out in a particular fashion  on a group $G$ and has infinite second moment so that $\mu$ is  expected to have the property that
$$\lim_{n\to +\infty}\frac{\Phi_\mu(n)}{\Phi_G(n)}=0,$$
how to get upper-bounds on $\Phi_\mu$ that are better than the universal upper-bound $\Phi_\mu(n)\le C_\mu\Phi_G(n/C_\mu)$ (assuming the support of $\mu$ generates $G$)?

In this work, we focus on finitely generated groups with polynomial volume growth, that is, groups for which there is an integer $d=d_G$ such that  (here $|A|$ denotes the number of elements in $A\subset G$) $$V(r)=|\{g\in G: |g|\le r\}|\asymp r^d.$$
On such a group, the function $\Phi_G$ satisfies $$\Phi_G(n)\asymp n^{-d/2}.$$ See \cite[Chapter VII]{VSCC} and the references therein. The results of \cite{BSClow,SCZlow} yield that, 
for any $\alpha\in (0,2)$ and any symmetric measure $\mu$ with finite weak-$\alpha$-moment,  there is a constant $c=c_\mu$
such that
$$ \Phi_\mu(n)\ge c_\mu n^{-d/\alpha}.$$
Also,  \cite{BSCsub,BSClow,CKSCWZ,SCZ-nil,SCZlow} provide examples of measures with finite 
weak-$\alpha$-moment for which a matching upper-bound can be obtained. The present work enlarges significantly the collection of such examples.  When thinking of such examples, it is important to remember that existing well established results imply that if $\mu_0$ is a symmetric probability measure satisfying, for all $n$,  $\Phi_{\mu_0}(n)\le C_0n^{-d/\alpha}$, then any other symmetric probability measure $\mu$ whose Dirichlet form,  
$$\mathcal E_\mu(f,f)=\frac{1}{2}\sum_{x,y\in G}|f(xy)-f(x)|^2\mu(y),$$
is such that
 $\mathcal E_{\mu_0}\le C\mathcal E_\mu$ will also satisfy $\Phi_\mu(n)\le C'n^{-d/\alpha}$.
 
 On a group $G$ with $V(r)\asymp r^d$, the simplest example of a $\mu$ with $\Phi_\mu(n)\asymp n^{-d/\alpha}$ is 
 \begin{equation}\label{defmualpha}
 \mu_\alpha(g)= \frac{c_{G,\alpha}}{(1+|g|)^{d+\alpha}} .
 \end{equation}
Note that this measure depends on the choice of the finite symmetric generating set $S$ that is implicit in the definition of the word-length $g\mapsto |g|$ and the constant $c_{G,\alpha}$ also depends on the choice of $S$.

\subsection{New results} In this subsection, we describe in simple terms two classes of measures 
for which we are able to prove sharp upper bounds on the probability of return at time $n$ which complement existing lower bounds. Later in the paper we put these results in a more general context and explain further consequences of these bounds. The new classes of measures considered here generalize and complement those studied in \cite{CKSCWZ,CKSCWZbook,SCZ-nil}.

Let $G$ be a finitely generated group with identity element $e$, equipped  with a symmetric finite generating set $\Sigma$. Let $|g|$ be the length of the element $g$ over $\Sigma$ (the smallest number of generators needed to write $g$; by convention, $|e|=0$). 

The first class of symmetric probability measures $\mu$ on $G$ we study is the family of all symmetric measures $\mu$ satisfying Condition ($\mathcal{L}_\alpha$), which reads as follows:
\begin{description}
\item [($\mathcal{L}_\alpha$)] for some fixed constants $A > 1$ and $\epsilon > 0$, for any integer $k\in \mathbb{N}$, 
\begin{itemize}
    \item there exists a subset $M_k$ of $G$ with 
    \begin{itemize}
        \item $ \mbox{card}(M_k)
   \succcurlyeq A^{kd}$
        \item  $|g| \asymp A^k$ for all $g \in M_k$, and
    \end{itemize}
    \item   $\mu(g) \ge \epsilon A^{-k(\alpha + d)}$ on $M_k \cup M_k^{-1}$
\end{itemize} 
\end{description}

Here $\mbox{card}(M_k)$ denotes the cardinality of the subset $M_k$, and 
$M_k^{-1} = \{g^{-1}: g\in M_k\}$. Measures satisfying ($\mathcal{L}_\alpha$) are generalizations of those given by Equation \ref{defmualpha} which uniformly assign a mass of roughly $|g|^{-d-\alpha}$ to every $g \in G$. Indeed, Condition ($\mathcal{L}_\alpha$)
only requires a  lower bound 
on a subset with the same volume growth rate as the balls of the group $G$; it leaves  open the possibility that the measure $\mu$ is small or vanishes on a large portion of any annulus at any scale.  

As a concrete example, 
imagine that each $M_k=B_{g_k}(r_k)$, $k\ge 1$, is a ball centered at a point $g_k\in G, |g_k|\le A^k$, and of radius $r_k\asymp A^k$.
We call ($B\mathcal{L}_\alpha$) this more restrictive version of ($\mathcal{L}_\alpha$), namely, 
\begin{description}
\item[($B\mathcal{L}_\alpha$)] for some fixed constants $A > 1$, $\epsilon>0$ and for any $k\in \mathbb{N}$,
\begin{enumerate}
    \item  
    there exists a ball $$M_k :=B_{g_k}(r_k)=\{g_k y:  |y|\le r_k\}$$ of radius $r_k \asymp A^k$ around $g_k \in G$ with $|g_k| \le A^k$, and
    \item $\mu(g) \ge \epsilon A^{-k(\alpha + d)}$ on $M_k \cup M_k^{-1}$
\end{enumerate}
\end{description}
\begin{theo}

\label{thm:1} Let $\alpha\in (0,2)$.    Let $G$ be a finitely generated group with a generating set $\Sigma$ such that $V(r) \asymp (1+r)^d$, i.e., $G$ has polynomial volume growth of degree $d$. Let $\mu$ be a symmetric probability measure satisfying ($\mathcal{L}_\alpha$). Then
    $$\mu^{(n)}(e) \preccurlyeq n^{-\frac{d}{\alpha}}.$$
\end{theo}

If the  measure $\mu$ further has finite weak $\alpha$-moment (ref. Definition \ref{def:weak}), 
a complementary lower bound 
$$\mu^{(n)}(e) \succcurlyeq n^{-\frac{d}{\alpha}}$$
follows from results in \cite{BSClow,SCZlow}. 
A sufficient condition for a measure $\mu$ to have  finite weak $\alpha$-moment is to admit the  upper bound $\mu(g)\preccurlyeq (1+|g|)^{-\alpha+d}$.  \\

The second type of symmetric probability measures we study are associated with the choice of a $k$-tuple of elements, $S=(s_1,\dots,s_k)$ and the map $$\pi_S:\mathbb Z^k\to G, \bar{a}=(a_1,\dots,a_k)\mapsto \pi_S(\bar{a})=s_1^{a_1}\dots s_k^{a_k}.$$ Namely, given a probability measure  $\psi$ on $\mathbb Z^k$, we define a symmetric probability measure, $\nu_{\psi,S}$, on $G$ by setting  
\begin{align}
 \nu_{\psi,S}(h)=   \label{nupsi}
    \begin{cases}
     \frac{1}{2}\sum\limits_{ \bar{a}: 
      \pi_S(\bar{a})\in \{h,h^{-1}\} } \psi(\bar{a}) &\text{if }\{h,h^{-1}\}\cap \pi_S(\mathbb Z^k)\neq \emptyset \\
         0 &\text{otherwise }\\
\end{cases}
\end{align}

We refer to this type of example as coordinate-wise stable-like measures even so the map $\pi$ is not, in general, a coordinate system for the group $G$. Indeed, $S$ may not contain enough elements to provide a coordinate system, i.e., no surjectivity is assumed for $\pi_S$, and no injectivity either. However, in some examples we will discuss later the tuple $S$ and the map $\pi$ will be related to the choice of a coordinate system for $G$. Note the we have built symmetry in the definition of $\nu_{\pi,S}$, independently of whether or not $\psi$ itself is symmetric as a measure on $\mathbb Z^k$.

Our main result for such measures is phrased in terms of the following pseudo-Poincar\'e inequality. Let $|h|_S$ be the word-length of an element $h$ of the subgroup of $G$ generated by the $k$-tuple $S$, $\langle S\rangle$, with respect to its symmetric generating finite set $\{s_1^{\pm1},\dots,s_k^{\pm 1}\}$. For $\bar{a}\in \mathbb Z^k$, set $\|\bar{a}\|=\sqrt{\sum_1^k|a_i|^2}$.

\begin{theo}\label{thm:2} Let $\alpha\in (0,2)$. Let $G$ be a finitely generated group. Let $\psi$ be a probability distribution on $\mathbb Z^k$ such that $\psi(\bar{a})\succcurlyeq (1+\|\bar{a}\|)^{-\alpha-k}.$ 
\begin{itemize}
    \item There is a constant $C=C(G,S,\psi)$ such that, for any $s\in S$, $n  \in \mathbb{N}$ and finitely supported function $f$,
\begin{equation}\label{PPh2} \sum_{g\in G}|f(gs^n)-f(g)|^2\le C n^\alpha \mathcal E_{\nu_{\psi,S}}(f,f).
\end{equation}
\item If $G$ is nilpotent, for all $h \in \langle S\rangle$ and finitely supported function $f$,
\begin{equation*} \sum_{g\in G}|f(gh)-f(g)|^2\le C |h|_S^\alpha \mathcal E_{\nu_{\psi,S}}(f,f).
\end{equation*}
\item If $G$ is nilpotent and the $k$-tuple $S$ generates $G$ as a group, 
$$\nu_{\psi,S}^{(n)}(e) \preccurlyeq n^{-\frac{d}{\alpha}}$$
and if moreover $\psi(\bar{a})\asymp (1+\|\bar{a}\|)^{-\alpha-k}$, then
$\nu_{\psi,S}^{(n)}(e) \asymp n^{-\frac{d}{\alpha}}.$ 
\end{itemize}
\end{theo}

The upper bound for $\nu_{\psi,S}^{(n)}(e)$ in the last bullet follows from the usual Nash inequality argument once the second bullet is established. (see, e.g., \cite[Appendix]{SCIso}, \cite[Chapters VI and VII]{VSCC} and \cite{CNash}). The lower bound  in the last statement  follows from  \cite{BSClow,SCZlow} and the fact that, in the considered cases, $\nu_{\psi,S}$ has finite weak $\alpha$-moment.

\begin{rem}
\label{rem: polypp}
Suppose $G$ is of polynomial growth, but not necessarily nilpotent. The pseudo-Poincar\'e inequality for a general $h \in \langle S\rangle$ and hence the convolution upper bound $\nu_{\psi,S}^{(n)}(e) \preccurlyeq n^{-\frac{d}{\alpha}}$ in the second and third bullet points in above theorem 
do not hold true as stated here but a modified version providing sharp results will nonetheless be obtained. We defer the discussion of this case to Section \ref{sec: polygrowth}
\end{rem}

\section{Pseudo-Poincar\'e inequalities}
\subsection{Measures satisfying ($\mathcal{L}_\alpha$)}

Let $G$ be a finitely generated group 
with a generating set $\Sigma$ 
such that $V(r) \asymp r^d$. Denote by $|g|$ the word length of $g\in G$ with respect to $\Sigma$.{\if Let $\mu$ be a symmetric probability measure that satisfies   ($\mathcal{L}_\alpha$):
\begin{description}
\item [($\mathcal{L}_\alpha$)] for some fixed constants $A > 1$ and $\epsilon > 0$, for any integer $k \ge 1$, 
\begin{itemize}
    \item there exists a subset $M_k$ of $G$ with 
    \begin{itemize}
        \item $\mbox{card}(M_k)
   \succcurlyeq A^{kd}$
        \item  $|g| \asymp A^k$ for all $g \in M_k$, and
    \end{itemize}
    \item   $\mu(g) \ge \epsilon A^{-k(\alpha + d)}$ on $M_k \cup M_k^{-1}$
\end{itemize} 
\end{description}
\fi}
The goal of this subsection is to prove the following theorem. 
\begin{theo}
\label{thm:point_pp}
Let $G$ be defined as above and let $\mu$ be a symmetric probability measure that satisfies ($\mathcal{L}_\alpha$). 
There exist a constant $C=C(G,\mu)$ and, for every integer $k\ge 1$, 
   a subset $G_k$ of $G$ such that  
   \begin{itemize}
    \item $\mbox{\em card}(G_{k}) \succcurlyeq A^{kd}$, 
    \item for every $h \in G_k$,  $|h| \asymp A^k$, and 
    \item for any $h \in G_k$ and finitely supported function $f$,
    \begin{align}
    \label{eq:point_pp}
    \sum_{x\in G}|f(xh)-f(x)|^2\le C |h|^\alpha \mathcal E_\mu(f,f).
\end{align}

\end{itemize}

\end{theo}
Before proving this result, we discuss how it leads to Theorem \ref{thm:1}.

\begin{proof}[Proof of Theorem \ref{thm:1}]
By a well-known argument, Theorem \ref{thm:1} follows from Theorem \ref{thm:point_pp}. Namely, the pseudo-Poincar\'e inequality in Equation (\ref{eq:point_pp}) and the volume lower bound of $G_k$ 
 imply the Nash inequality (see, e.g., \cite[Appendix]{SCIso})
 $$\|f\|_2^{2+2\alpha/d}\le C\mathcal E_\mu(f,f)\|f\|_1^{2\alpha/d}$$
 In turn, this Nash inequality implies the upper bound $\mu^{(n)}(e)\preccurlyeq n^{-d/\alpha}.$ See, e.g., \cite[Chapters VI and VII]{VSCC} and \cite{CNash}.\end{proof}

To prove Theorem \ref{thm:point_pp} we first find the subsets $G_k$, $k\in \mathbb{N} $, referenced in Theorem \ref{thm:point_pp} and explore their properties, which lead to the pseudo-Poincaré inequality. 
\begin{lem}
\label{lem: rep}
Let $G$ be defined as above and let $\mu$ be a symmetric probability measure that satisfies ($\mathcal{L}_\alpha$). For each $k \in \mathbb{N}$, 
    there exists
    a subset $G_{k}$ of $G$ with the following properties:
\begin{itemize}
    \item $\mbox{\em card}(G_{k}) \succcurlyeq A^{kd}$
    \item for any $h \in G_{k}$, $|h|\asymp A^k$
    \item for any $h \in G_{k}$, the tuple set $$ R(h)  = \{(g,g') \in M_k^{-1} \times M_k: gg' = h\}$$ satisfies 
    $\mbox{\em card}(R(h)) \succcurlyeq  A^{kd}$
\end{itemize}
\end{lem}
\begin{rem} The relationship between the sequence of sets $G_k$ and the measure $\mu $ which satisfies ($\mathcal{L}_\alpha$) is encoded in the third bullet condition of Lemma \ref{lem: rep}.\end{rem}
\begin{proof}
Fix some $k \in \mathbb{N}$. 
Take
\begin{align*}
    T = T_k & := M_k^{-1} \times M_k = \{(g,g') : g\in M_k^{-1}, g'\in M_k\}\\
    P = P_k & := \{gg': (g,g') \in T)\} \subseteq G
\end{align*}
Since elements of $P$ are of word length no greater than $2A^k$, $\mbox{card}(P) \preccurlyeq (2A^k)^d$. Recall by hypothesis $\mbox{card}(M_k) \succcurlyeq A^{kd}$. Hence we can find a constant $c$ with  $\mbox{card}(P) \le c \cdot  \mbox{card}(M_k)$. For each $h \in P$, define 
\begin{align}
\label{eq: preimage}
    R(h)  = \{(g,g') \in M_k^{-1}\times M_k: gg' = h\}
\end{align}
Observe that $\mbox{card}(R(h)) \le \mbox{card}(M_k) $ for any $h \in P$. Indeed, if $\mbox{card}(R(h)) > \mbox{card}(M_k)$, that would imply $h = gg_1 = gg_2$ for some $g \in M_k^{-1}$ and $g_1 \ne g_2 \in M_k$ which is not possible. It follows $P$ can be partition as $\Tilde{G}_k \cup (P\backslash \Tilde{G}_k)$ where
\begin{align*}
      \Tilde{G}_{k}  &= \left\{h \in P :  \frac{1}{2c}\mbox{card}(M_k)  \le \mbox{card}(R(h)) \le \mbox{card}(M_k)\right \}  \\
      P\backslash \Tilde{G}_k & =  \left\{h \in P :  \mbox{card}(R(h)) < \frac{1}{2c}\mbox{card}(M_k)\right \} 
\end{align*}
This gives an upper bound for $\mbox{card}(T)$: 
\begin{align*}
    \mbox{card}(T)  &\le  \mbox{card}(M_k) \mbox{card}(\Tilde{G}_k) + \frac{1}{2c}\mbox{card}(M_k) \mbox{card} (P \backslash \Tilde{G}_k)   \\
    & \le \mbox{card}(M_k)\mbox{card}(\Tilde{G}_k) +  \frac{1}{2c} \mbox{card}(M_k) \mbox{card}(P) \\
     & \le \mbox{card}(M_k)\mbox{card}(\Tilde{G}_k) +  \frac{1}{2} \mbox{card}(M_k)^2
\end{align*}

As $\mbox{card}(T)  =  \mbox{card}(M_k) ^2$, it follows  $\mbox{card}(\Tilde{G}_k) \ge \frac{1}{2}\mbox{card}(M_k) \succcurlyeq \frac{1}{2} A^{kd}$. 
Find $j \ge 1$ large enough such that  
$\mbox{card}(B_e(A^{k-j})) <   \frac{1}{2}  \mbox{card}(\Tilde{G}_k)$.
Let $G_k$ be the intersection 
 of $\Tilde{G}_k$ and the annulus 
$B_e(A^{k})-B_e(A^{k-j})$; it's of size at least $ \frac{1}{2}  \mbox{card}(\Tilde{G}_k)$. 
Trivially, one can verify $G_k$ 
satisfies all the desired properties, completing the proof of the lemma. 
\end{proof}

\begin{proof}[Proof of Theorem \ref{thm:point_pp}]

Fix $k \in \mathbb{N}$ and let $G_k$ be defined as in Lemma \ref{lem: rep}. Take $h \in G_k$ and a finitely supported function $f$, it remains to check that Equation \ref{eq:point_pp} holds. 

 Let $R(h)$ be defined as in Equation \ref{eq: preimage}. Observe that by the definition of $R(h)$, for any $(x_1, x_2)\ne  (x_3, x_4) $ in $R(h)$, it must be that $x_1 \ne x_3$ and $x_2 \ne x_4$. 
Indeed, if, for instance,  $x_1 = x_3$, then $h = x_1x_2 = x_1x_4$  implies $x_2 = x_4$, a contradction. Let $$R_1(h) = \{x \in M_k^{-1} : \exists x' \in M_k \text{ such that }(x,x') \in R(h) \}
 $$ It follows that $\mbox{card}(R_1(h)) = \mbox{card}(R (h))$.  
 For each $x \in R_1(h)$ with $(x,x') \in R(h)$, we set $R_2(h,x) = x'$. Again,  by our earlier observation,  $R_2(h,\cdot)$ is well-defined. 
Let $\nu$ be the (sub-probability) measure that's equal to $\epsilon A^{-k(\alpha + d)}$ on $M_k \cup M_k^{-1}$ and 0 otherwise. Compute
    \begin{align*}
        &\sum_{g \in G} |f(gh) - f(g)|^2 \sum_{
       x\in R_1(h)} \nu(x) \\
         &\le 2 \sum_{
        g \in G, x\in R_1(h) } |f(gh) - f(gx)|^2  \nu(x) +  
         2 \sum_{
        g \in G, x\in R_1(h) }|f(gx) - f(g)|^2  \nu(x) \\
         &= 2 \sum_{
        g \in G, x\in R_1(h) } |f(gx^{-1}h) - f(g)|^2  \nu(x) +  
         2 \sum_{
        g \in G, x\in R_1(h) }|f(gx) - f(g)|^2  \nu(x) \\
        &\le 2 \sum_{
          g \in G, x\in R_1(h) } |f(gR_2(h,x)) - f(g)|^2  \nu(x) +  
         2\sum_{ g \in G, x\in R_1(h)} |f(gx) - f(g)|^2  \nu(x) \\
         & \le 2 \mathcal{E}_\nu(f,f)   + 2 \mathcal{E}_\nu(f,f) \preccurlyeq 4 \mathcal{E}_\mu(f,f) 
    \end{align*}
    In the second last step, we use the fact that $R_2(h,x) \in M_k$ is assigned the same measure  by $\nu$ as $x \in R_1(h)$, and the last steps follows from $\nu(g) \preccurlyeq \mu(g)$ for all $g \in G$. 
    It remains to compute
    $$\sum_{
       x\in R_1(h)} \nu(x) = \epsilon A^{-k(\alpha + d)}\cdot  \mbox{card}(R(h))  \succcurlyeq \epsilon A^{-k\alpha} \asymp  |h|^{-\alpha}$$
Here we used the last two statements regarding $h \in G_k$ in Lemma \ref{lem: rep}, i.e. $|h|\asymp A^k$ and $\mbox{card}(R(h))  \succcurlyeq A^{kd}$. 
\end{proof}

\subsection{Measures satisfying ($B\mathcal{L}_\alpha$)}
 In this subsection, we consider a special case of condition ($\mathcal{L}_\alpha$) which we call ($B\mathcal{L}_\alpha$): 
\begin{description}
\item[($B\mathcal{L}_\alpha$)] for some fixed constants $A > 1$, $\epsilon>0$ and for any $k\in \mathbb{N}$,
\begin{enumerate}
    \item  
    there exists a ball $$M_k :=B_{g_k}(r_k)=\{g_k y:  |y|\le r_k\}$$ of radius $r_k \asymp A^k$ around $g_k \in G$ with $|g_k| \le A^k$, and
    \item $\mu(g) \ge \epsilon A^{-k(\alpha + d)}$ on $M_k \cup M_k^{-1}$
\end{enumerate}
\end{description}

For such measures, the subsets $G_k$ referenced in Lemma \ref{lem: rep} can be described explicitly as follows: 
\begin{align*}
    G_1 &=  B_e(A) \\
    G_k & =  B_e(r_k/2) - B_e(r_k/c_0), \quad k \ge 2 
\end{align*}
for some $c_0 > 2$. They satisfy the first two statement in Lemma \ref{lem: rep} trivially. To verify the last property, fix $k\in \mathbb{N}$ and take $h \in G_k$. Note that 
    \begin{align*}
        R(h) & = \{( x^{-1}g_k^{-1}, g_k y): |x| , |y| \le r^k \text{ and }x^{-1}y = h\} \\
        & \supseteq \{( x^{-1}g_k^{-1}, g_k y ): |x| \le r^k/4 \text{ and }x^{-1}y = h\} 
    \end{align*}
It follows $\mbox{card}(R(h)) \ge B_e(r_k/4) \asymp A^{kd}$ as desired. Since $\bigcup_{k\in \mathbb{N}} G_k = G$, in this case, a stronger version of Theorem \ref{thm:point_pp} holds. 

\begin{theo}
\label{thm:ball_pp}
Let $G$ be defined as above and let $\mu$ be a symmetric probability measure that satisfies ($B\mathcal{L}_\alpha$). 
There exists a constant $C=C(G,\mu)$ 
such that for any $h \in G$ and finitely supported function $f$,
\begin{align*}
    \sum_{x\in G}|f(xh)-f(x)|^2\le C |h|^\alpha \mathcal E_\mu(f,f).
\end{align*}

\end{theo}

Another  simple observation is that the support of $\mu$ generates $G$. Namely, we have
\begin{equation}\mu^{(2)}(g)\succcurlyeq A^{-(2\alpha+d)j} \mbox{ for } g\in B_e(r_j/2).
\end{equation}
To see this  simply write
$$\mu^{(2)}(g)\ge \sum_{x^{-1}\in B_{g_j}(r_j/2)}\mu(x^{-1}g)\mu(x)\succcurlyeq A^{jd} A^{-2(\alpha+d)j}=A^{-(2\alpha +d)j}.$$
Certainly, for $j$ large enough, $\Sigma\subset B_e(r_j/2)$ and thus $\mu^{(2)}(g)>0$ on $\{e\}\cup\Sigma$ as desired. This property still holds if we relax the condition (($B\mathcal{L}_\alpha$)) by asking that it holds only for $k$ large enough. By the same token, $\mu^{(3)}(e)>0$ which shows that $\mu$ is aperiodic and its $L^2$-spectrum  is contained in $[-1+\eta,1]$.

\subsection{Coordinate-wise stable-like measures}
Fix a $k$-tuple $S$ of elements in $G$, $S=(s_1,\dots,s_k)$ and define the map $\pi_S:\mathbb Z^k\to G, \bar{a}=(a_1,\dots,a_k)\mapsto \pi_S(\bar{a})=s_1^{a_1}\dots s_k^{a_k}$. Given a probability measure  $\psi$ on $\mathbb Z^k$, we define a symmetric probability measure, $\nu_{\psi,S}$, on $G$ by (\ref{nupsi}). Here we are focusing on the case when 
$\psi(\bar{a})\asymp (1+\|\bar{a}\|)^{-\alpha-k}$.
The collection of all such measures is denoted by $M_{S,\alpha}(G)$.  The aim of this section is to prove Theorem \ref{thm:2}.
In a later section, we discuss finite convex combinations of such measures when both $S$ and $\alpha$ are allowed to vary. This is a significant generalization/variation on results contained in \cite{SCZ-nil,CKSWZ1}. In \cite{SCZ-nil}, the case when $S$ is a singleton, i.e., $S=(s)$, is treated (including convex combinations of such measures). In \cite{CKSWZ1}, finite
convex combination of measures of the type $\nu_{H,\alpha}(h)\asymp (1+|h|)^{-\alpha-d_H}$ where $H$ is a subgroup of $G$ and both $H$ and $\alpha$ are allowed to vary is treated.

Before embarking with the proof of Theorem \ref{thm:2}, observe that the condition $\psi(\bar{a})\preccurlyeq (1+\|\bar{a}\|)^{-\alpha-k}$ implies that the probability measure $\nu_{\psi,S}$ has finite $\alpha$ weak-moment on $\langle S\rangle$. This is because $|\pi(\bar{a})|_S\le \sum_1^k|a_i|\le \sqrt{k}\|\bar{a}\|$ and thus $$t^\alpha \nu_{\psi,S}( |h|>t)\le t^\alpha \sum_{\|\bar{a}\|>t/\sqrt{k}}\psi(\bar{a})\preccurlyeq 1.$$
This means that the lower bound on $\nu_{\psi,S}^{(2n)}(e)$ in Theorem \ref{thm:2} follows from the results of \cite{BSClow,SCZlow}.

\begin{exa}
Take $G$ to be the dimension-4 unipotent matrix group. A natural choice is to take $S$ to be the Mal'cev basis 
$$(M_{12}, M_{23} ,M_{13} , M_{34} , M_{24} ,M_{14} )$$
where $M_{ij}$ is the matrix with all entries set to zero except for a 1 in the $(i,j)$ position and along the diagonal. (See Appendix \ref{app: Malcev}, in particular Example  \ref{exa: matrix_coord}, for more discussion on Mal'cev basis.) Then the map $\pi_S$ gives the matrix coordinate system, i.e., 
\begin{align*}
    \pi_S(a_{14}, a_{24}, a_{34}, a_{13},a_{23},a_{12}) & = 
    M_{14}^{a_{14}}M_{24}^{a_{24}} M_{34}^{a_{34}} M_{13}^{a_{13}}M_{23}^{a_{23}}M_{12}^{a_{12}}\\
    &= \begin{pmatrix}
    1 & a_{12} & a_{13} & a_{14}\\
    0 & 1 & a_{23} & a_{24}\\
    0 & 0 & 1& a_{34} \\
    0 & 0 & 0 & 1
\end{pmatrix} 
\end{align*}
The coordinate-wise $\alpha$-stable measure  associated with the choice 
$$\psi(\bar{a})=\frac{c_\alpha}{(1+\|\bar{a}\|_2^2)^{(6+\alpha)/2}},
\bar{a} \in \mathbb Z^6, \;\|\bar{a}\|_2^2=\sum|a_{ij}|^2,$$ is approximately given by 
\begin{align*}
   \nu_{\psi, S}(\pi_S(a_{14}, a_{24}, a_{34}, a_{13},a_{23},a_{12} ))  
    \asymp   \frac{c_\alpha}{(1+\|\bar{a}\|_2^2)^{(6+\alpha)/2}} + \frac{c_\alpha}{(1+\|\bar{a}'\|_2^2)^{(6+\alpha)/2}}
\end{align*}
where
$$
\bar{a}' = (-a_{12}a_{23}a_{34} + a_{13}a_{34} - a_{14}, a_{23}a_{34} - a_{24}, -a_{34},  a_{12}a_{23} - a_{13},-a_{23},-a_{12})
$$
satisfies that $\pi_S(\bar{a})^{-1} = \pi_S(\bar{a}')$. 
\end{exa}

Below, we present two basic algebraic facts that will be frequently used in the subsequent proofs.
\begin{lem}
\label{lem: decomp}
For any $h \in G$, 
$$\sum_{g\in G} |f(gh) - f(g)|^2  = \sum_{g\in G} |f(gh^{-1}) - f(g)|^2 $$
and furthermore if it can be decomposed as a product $h=\prod_{i=1}^k h_i$, $$\sum_{g\in G} |f(gh) - f(g)|^2 \le  
   k \sum_{i=1}^k \sum_{g\in G}|f(gh_i) - f(g)|^2 $$
\end{lem}
\begin{proof}
    The first equality is a result of the Cayley's theorem, i.e. $gG = G$ for every $g \in G$. The second inequality follows from Cauchy Schwartz inequality, Cayley's theorem and a telescoping sum argument. 
\end{proof}

\begin{proof}[Proof of Theorem \ref{thm:2}]
    In this proof, We may omit the subscripts in $\nu_{\psi_\alpha, S}$ and $\psi_\alpha$. Without loss of generality, we  assume $\psi(\bar{a})\asymp (1+\|a\|)^{-\alpha-k}$ with $\alpha\in (0,2)$. First, we concentrate on proving the desired 
   pseudo-Poincar\'e inequality for element $h\in \langle S\rangle$ of the type $h=s_i^a$ in the form
   \begin{equation}\label{PPa}
       \sum_{g\in G}|f(gs_i^a)-f(g)|^2\le C |a|^\alpha \mathcal E_\nu(,f).
   \end{equation}

    For any $a \in \mathbb{Z}$ and any $\Bar{a}  = (a_1,\ldots, a_k)  \in \mathbb{Z}^k $, $s_1^a = s_1^a\pi(\bar{a}) \cdot \pi(\bar{a})^{-1}$, so by Lemma \ref{lem: decomp}, 
$$\sum_{g} |f(g s_1^{a}) - f(g)|^2 \le 2\sum_{g} |f(gs_1^a\pi(\bar{a})) - f(g)|^2 +2\sum_{g} |f(gs_1^a\pi(\bar{a})) - f(g)|^2 $$

We write $\Bar{a} \ge a$ to signify that $|a_i| \ge a$ for $i=1,\ldots, k$. The above inequality gives
\begin{align*}
    \sum_{g} &|f(g s_1^{a}) - f(g)|^2   \sum_{\Bar{a} \ge 2a} \psi(\Bar{a}) \\
    &  \le  2\sum_{\Bar{a} \ge 2a}  \sum_{g} |f(g ) - f(gs_1^{-a}\pi(\Bar{a}))|^2 \psi(\Bar{a})  +  2\sum_{\Bar{a} \ge 2a} \sum_{g} |f(g) - f(g\pi(\Bar{a}))|^2\psi(\Bar{a}) \\
    & =: 2(I_1 + I_2)
\end{align*}
Obviously we have $\psi(\Bar{a}) \le 2\nu(\pi(\Bar{a}))$ hence
$$I_2 \le 2\sum_{|\Bar{a}| \ge 2a} \sum_{g} |f(g) - f(g\pi(\Bar{a}))|^2 \nu(\pi(\Bar{a})) \preccurlyeq  \mathcal{E}_\nu(f,f).$$
Similarly, because $|\Bar{a}| \ge 2a$, we also have
$\psi(\Bar{a}) \le  2\psi(a_1-a, \ldots, a_n) \le 4\nu(s_1^{-a} \pi(\Bar{a}))$
and 
$$I_1 \le4\sum_{|\Bar{a}| \ge 2a}  \sum_{g} |f(g ) - f(gs_1^{-a}\pi(\Bar{a}))|^2 \nu(s_1^{-a }\pi(\Bar{a})) \preccurlyeq \mathcal{E}_\nu(f,f) $$
The desired inequality then follows from the following estimate
$$\sum_{|\Bar{a}| \ge 2a} \psi(\Bar{a}) \asymp \sum_{|\Bar{a}| \ge 2a} (1 + |\Bar{a}|)^{-\alpha + n} \asymp a^{-\alpha}.$$
We now proceed by induction on $i\in\{1,\dots,k\}$. Assume that 
$$\sum_{g} |f(gs_i^a) - f(g)| \le C |a|^{\alpha} \mathcal{E}_{\nu}(f,f)$$
for $i = 1,\ldots, j-1$. For every $a\in \mathbb{Z}$ and $1 \le i \le k$, define the following two subsets of $\mathbb{Z}^k$: 
\begin{align}
    U_{a,i} &:=\{
    ( \underbrace{0, \ldots,0}_{(i-1)-times}, a_i, \ldots, a_k )
   \;:\;|a_i|,|a_{i+1}|,\ldots,|a_k| \ge 2a\} \label{eq: u}\\
  D_{a,i} &:=\{
 ( a_1, \ldots, a_i, 
 \underbrace{0, \ldots,0}_{(k-i)-times} )\;:\;|a_i|,|a_{i-1}|,\ldots,|a_1| \ge 2a\} \label{eq: d}
\end{align}
For simplicity, we may also use the notation $(\textbf{0}_{i-1}, a_i, \ldots, a_k)$ to denote the element $( \underbrace{0, \ldots,0}_{(i-1)-times}, a_i, \ldots, a_k )$. Furthermore, the sum of elements in $U_{a,i}$ and $D_{a,i}$ is defined component-wise in the usual way. 

For any $a \in \mathbb{Z}$ and $u\in U_{a,j} ,d \in D_{a,j-1}$, 
$$s_j^a = \pi(d)^{-1} \cdot \pi(d) s_j^a \pi(u) \cdot \pi(u)^{-1} \pi(d)^{-1} \cdot \pi(d)$$
so 
\begin{align*}
    \sum_{g} |f(gs_j^{a}) - f(g)|^2   & \preccurlyeq  8  \sum_{g} | f(g) -f (g\pi(d))|^2 +  4\sum_{g} |f(g\pi(d)\pi(u)) - f(g) |^2  \\
       & \quad \quad + 4\sum_{g}| f(g \pi(d)s_j^{a}\pi(u))- f(g )|^2
\end{align*}

By the induction hypothesis, the first term is bounded above by some constant multiple of $|a|^\alpha\mathcal{E}_{\nu}(f,f)$, so   
\begin{align*}
    & \left(\sum_{g} |f(gs_j^{a}) - f(g)|^2\right)  \left(\sum_{
  (u,d)\in U_{a,j}\times D_{a,j-1}}
   \psi(u+d)  \right)  \\
   & \le  8 |a|^\alpha\mathcal{E}_{\nu}(f,f)  \left(\sum_{(u,d)\in U_{a,j}\times D_{a,j-1}    }
   \psi(u+d)  \right) \\
   & \quad + 4 \sum_{(u,d)\in U_{a,j}\times D_{a,j-1}}    \sum_{g}| f(g \pi(d)\pi(u))- f(g )|^2    \psi(u+d) \\
    & \quad + 4  \sum_{ (u,d)\in U_{a,j}\times D_{a,j-1}}    \sum_{g} |f(g\pi(d)s_j^{a}\pi(u)) - f(g) |^2  \psi(u+d)  =: I_1 + I_2 + I_3
   \end{align*}
Note that in $I_3$, the element $\pi(d)s_j^{a}\pi(u)$ is of the form
\begin{align*}
  \pi(a_1, \ldots, a_{j-1}, \textbf{0}_{k-j+1}) s_j^a  \pi(\textbf{0}_{j-1}, a_j,\ldots, a_k)  = s_1^{a_1}\ldots s_{j-1}^{a_{j-1}} s_j^{a+a_j} \ldots s_k^{a_k}
\end{align*}
Clearly, for any $(u,d)\in U_{a,j}\times D_{a,j-1}$, we  have 
$$\psi(u+d) \le 2 \psi(a_1,\ldots, a_{j-1}, a_j+a,a_{j+1},\ldots, a_k) \le 4\nu(\pi(u)^{-1}s_j^{a}\pi(d)) $$
With the same argument, in $I_2$, we can replace $\psi(u+d)$ with $\nu(\pi(d)\pi(u))$. We deduce $I_2 + I_3$ are bounded above by $\mathcal{E}_\nu(f,f)$. The desired inequality again follows from the observation that $\sum_{
    (u,d)\in U_{a,j}\times D_{a,j-1} }
   \psi(u+d) \asymp |a|^{-\alpha}$.
   This proves inequality  (\ref{PPa}). To obtain the pseudo-Poincar\'e inequality (\ref{PPh2} stated in Theorem \ref{thm:2}, we need an additional argument because it is often the case that $|\pi(\bar{a})|_S$ is much smaller than $\|\bar{a}\|$. For instance, assume that 
   $s_3=[s_1,s_2]=s_1s_2s_1^{-1}s_2^{-1}$
   is the commutator of $s_1$ and $s_2$. Then it is well understood that
   $|s_3^a|_S\preccurlyeq \sqrt{|a|}$. Similarly, if $s_3=[s_1,[s_1,s_2]]$ then
   $|s_3^a|_S\preccurlyeq |a|^{1/3}$, etc. See \cite[Proposition 2.17]{SCZlow} for a very general statement. Fortunately, the result we need to go from (\ref{PPa}) to (\ref{PPh2}) is exactly \cite[Theorem 2.10]{SCZlow}. 
\end{proof}

\section{Applications and variations}

\subsection{Further properties}
In this section we described further results concerning the convolution power of symmetric probability measures $\mu$ which have finite weak-$\alpha$-moment condition  and satisfy either ($B\mathcal{L}_\alpha$) or are of the type (\ref{nupsi}) when the $k$-tuple $S=(s_1,\dots,s_k)$ generates $G$  and $\psi(\bar{a}\asymp(1+\|a\|)^{-\alpha-k}$, $\alpha\in (0,2)$.  
What these measures have in common is that they satisfy 
\begin{enumerate}
    \item The $\alpha$ weak-moment condition
    $$\sup_{s>0} \{s^\alpha \mu(\{g:|g|\ge s\})\}<+\infty;$$
    \item The pseudo-Poincar\'e inequality
\begin{equation}\label{PPmu}
\sum_{g\in G}|f(gh)-f(g)|^2\le C_\mu |h|^\alpha   \mathcal E_\mu(f,f), h\in G.\end{equation}
\end{enumerate}
We have already seen that, as a consequence, such measures satisfy $\mu^{(2n)}(e)\asymp n^{-d/\alpha}$.  The results of \cite{SCZlow,CKSWZ1}
 gives the following additional properties. We let $(X_n)_0^{+\infty}$ denote the random walk driven by the measure $\mu$, and by $\mathbf P_x$ the probability distribution of $(X_n)_0^{+\infty}$ started at $X_0=x$.
 \begin{theo} \label{thm:1b} Fix $\alpha\in (0,2)$. Let $G$ be a finitely generated group with a generating set $\Sigma$ such that $V(r) \asymp r^d$ and let $\mu$ be an irreducible symmetric probability measure satisfying the properties 1 and 2 formulated above. For simplicity, assume further that $\mu(e)>0$.
 \begin{itemize}
     \item There exists a constant $C$ such that, for all integers $n,m$ and group elements $x,y\in G$,
     \begin{equation}|\mu^{(n+m)}(xy)-\mu^{(n)}(x)|\le C \left(\frac{m}{n}+ \frac{|y|^{\alpha/2}}{\sqrt{n}}\right)\mu^{(2\lceil n/2\rceil)}(e).\end{equation}
     \item There exists $\eta$ such that for all $n\ge 1$ and $x\in G$ such that $|g|^\alpha\le \eta n$, it holds that
    \begin{equation}
        \mu^{(n)}(g)\asymp n^{-d/\alpha.}
    \end{equation}
    \item For all $\epsilon>0$ there exists $\gamma>0$ such that
    for all $n$,
    \begin{equation}
        \mathbf P_e\left(\sup_{k\le n}\{|X_k|\}\ge \gamma n^{1/\alpha}\right)
\le \epsilon.    \end{equation}
  \item There exists $\epsilon, \gamma_1 >0$ and $\gamma_2 \ge 1
  $ such for 
    for all integer $n, \tau $ with 
    $\frac{1}{2} \tau^\alpha/\gamma_1 \le n \le \tau^\alpha/\gamma_1$
    \begin{equation}
        \inf_{x:|x|\le \tau^\frac{1}{\alpha}}\left\{ \mathbf P_x\left(\sup_{k\le n}\{|X_k|\}\le \gamma_2 \tau ; |X_n|\le \tau\right)\right\}
\ge \epsilon\end{equation}    
 \end{itemize}
 \end{theo}
\begin{rem} If $\mu(e)=0$ but $\mu$ is aperiodic, the second bullet must be adjusted by replacing ``$n\ge 1$'' by ``$n\ge n_0$ for some $n_0$.'' One can take $n_0$ to be the smallest integer $m$ such that $\mu^{(m)}(e)>0$.
    \end{rem}
The first bullet is a sort of regularity estimate in time and space. It leads to the second bullet which is called a ``near-diagonal estimate.'' This is to be compared  with ``diagonal estimates'' (estimates for $\mu^{(n)}(e)$) and off-diagonal estimates (estimates of $\mu^{(n)}(x)$ for all $n$ and $x$ describing the decay to $0$ when $x$ tends to infinity in $G$). The last two bullets answer classical questions regarding exit times of balls for the associated random walk. They are discussed in \cite{SCZlow,CKSWZ1} under the names of ``control'' and ``strong control.'' See \cite[Proposition 1.4]{SCZlow} and \cite[Section 5]{CKSWZ1}.

\subsection{Variations }

Let us briefly discuss two types of interesting variations on measures satisfying finite weak-$\alpha$-moment and either ($B\mathcal{L}_\alpha$) condition or being of  coordinate-wise type (\ref{nupsi}).

\subsubsection*{Regular variation} First, let us observe that if a symmetric probability measure $\mu$ satisfies the finite weak-$\alpha$-moment and the ($B\mathcal{L}_\alpha$) condition 
or is a coordinate-wise measure with generating tuple $S$, and if in these conditions $\alpha>2$, then $\mu$ has finite second moment and satisfy
$$\sum_{g\in G}|f(gh)-f(g)|^2\le C_\mu |h|^2  \mathcal E_\mu(f,f), h\in G.$$ It follows that
$\mu^{(2n)}(e)\asymp n^{-d/2}$ for $n\ge 1$ and all the conclusions stated in Theorem \ref{thm:1b} holds for such measure with $\alpha$ replaced by $2$. 

Now,  let $\alpha\in (0,2)$ and $\ell: (0,+\infty)\to (0,+\infty)$ be a slowly varying function. Define Condition ($\mathcal{U}_{\alpha,\ell}$) by replacing the power function $t\mapsto t^\alpha$ in the finite weak-$\alpha$-moment condition by the regularly varying function $t\mapsto t^\alpha \ell(t)$. Define ($B\mathcal{L}_{\alpha,\ell}$) by replacing $$A^{-k(\alpha +d)} \mbox{ by }A^{-k(\alpha+d)}\ell(A^{-k})$$ in  ($B\mathcal{L}_{\alpha}$). For coordinate-wise measures, replace the condition
$$\psi(\bar{a})\asymp (1+\|\bar{a}\|)^{-\alpha-k} \mbox{ by
}\psi(\bar{a})\asymp \frac{1}{(1+\|\bar{a}\|)^{\alpha+k}\ell(1+\|\bar{a}\|)}.$$

The same techniques described above lead to sharp two-sided estimates for $\mu^{(2n)}(e)$ under such hypotheses. See \cite{BSClow,SCZ-nil,CKSWZ1} for the treatment of similar but different examples involving slowly varying functions. Precise statements involve dealing with various classical computation regarding the function $\ell$ and computing the inverse of $t\mapsto t^\alpha \ell(t)$.

\begin{rem} The case $\alpha=2$ is omitted above because it requires different arguments. The technique used here to obtain the key sharp pseudo-Poincar\'e does not provide the relevant sharp pseudo-Poincar\'e inequality in the case $\alpha=2$ and related regular variation cases.  
\end{rem}

\subsubsection*{Lack of symmetry}A second variation of interest concerns the non-symmetric case and we discuss it briefly. 
Let us replace  
($B\mathcal{L}_\alpha$)  by its non-symmetric version:
\begin{description}
    \item[] There exists a ball $M_k :=B_{g_k}(r_k)$ of radius $r_k \asymp A^k$ around $g_k \in G$ with $|g_k| \le A^k$ such that 
    $\mu(g) \ge \epsilon A^{-k(\alpha + d)}$ on $M_k$
\end{description}
A well-known approach to study upper-bounds on $\mu^{(n)}(e)$ when $\mu$ may not be symmetric is to consider the  multiplicative symmetrization $\check{\mu}*\mu$ where $\check{\mu}(x)=\mu(x^{-1})$. Indeed, a Nash inequality of the type
\begin{equation}\label{Ncheck} \|f\|_2^{2+2\alpha/d}\le C\mathcal E_{\check{\mu}*\mu}(f,f)\|f\|_1^{2\alpha/d}\end{equation} implies that \begin{equation}
\label{UBcheck} \sup_{g\in G}\{\mu^{(n)}(g)\}\le C(C,\alpha,d)n^{-d/\alpha}.\end{equation}

Because the measure $\check{\mu}*\mu$ is symmetric, in order to prove that it satisfies ($B\mathcal{L}_\alpha$) it suffice to show that
$$\check{\mu}*\mu(g)\ge \epsilon^2 A^{-k(\alpha+d)} \mbox{ on } B_{g_k}(r_k-r_0-1).$$
This is easy to see because, for $g\in B_{g_k}(r_k/2)$,
\begin{eqnarray*} \check{\mu}*\mu(g)&=& \sum_{x\in G}\check{\mu}(x^{-1}g)\mu(x)=  \sum_{x\in G}\check{\mu}(x)\mu(gx^{-1})\\
&=&\sum_{x\in B_0^{-1}} \mu(x^{-1})\mu(gx^{-1})\\
&\ge & \epsilon^{2} A^{-k(\alpha+d)}.
\end{eqnarray*}
(The measure $\check{\mu}*\mu$ also satisfies $\check{\mu}*\mu(e)>0$).  Applying Theorem \ref{thm:ball_pp} to $\check{\mu}*\mu$ gives a pseudo-Poincar\'e inequality which, in turns, together with the volume growth estimate, gives the Nash inequality (\ref{Ncheck}) and the upper-bound (\ref{UBcheck}).

The same argument applies to coordinate-wise measures where now we define the non-symmetric version $\mu_\psi$  of $\nu_\psi$ by
$$\mu_\psi(h)=\sum_{\bar{a}\in \mathbb Z^d: \pi_S(\bar{a})=h}\psi (\bar{a})$$
and we assume $\psi(\bar{a})\succcurlyeq (1+\|\bar{a}\|)^{\alpha-k}$ and $S$ is generating.  The measure $\mu_\psi$ satisfies $$\sup_{g\in G}\{\mu_\psi^{(n)}(g)\}\preccurlyeq n^{-d/\alpha}.$$

Obtaining a good lower bound for $\sup_{g\in G}\{\mu^{(n)}(g)\}$ in any of those non-symmetric cases is a subtle open question. In many cases, one expects that the upper-bounds described above  will not be sharp when $\mu$ is not symmetric.

\section{Multi-strength stable-like measures}

In this section, we will investigate a large family of convex combinations of stable-like probability measures and, more specifically, the properties that enable us to obtain estimates of the convolution powers of such measures. The arguments in this section make heavy use of the results in \cite{SCZ-nil,CKSCWZ}. The main result provides a useful variation and complement to \cite{CKSCWZ}. 

Recall that the word-length distance associated with a finite subset $\Sigma$
of $G$ is defined as 
$$|g|_{\Sigma}  = \inf \{m: \exists \omega \in (\Sigma \cup \Sigma^{-1})^{m}, g =\omega \text{ in }G\}$$
and if $g \not\in \langle \Sigma \rangle$, we set $|g|_\Sigma = \infty$.

In this section, we'll focus on symmetric measures that are convex combination of measures satisfying the properties ($\mathcal{U}_{\Sigma, \alpha}$) and ($\mathcal{PP}_{\Sigma, \alpha}$) defined below, where $\Sigma$ is a finite subset of $G$ and  $\alpha \in (0,2)$. The two parameters $\Sigma$ and $\alpha$ are allowed to vary in the given convex combination. These two properties are defined as follows:
     \begin{description}
    \item [($\mathcal{U}_{\Sigma, \alpha}$)] The measure $\nu$ satisfies $supp(\nu) \subseteq \langle \Sigma\rangle$ and the tail and truncated second moment have order no greater than $\alpha$, i.e. 
    $$ \sum_{|h|_\Sigma \ge t} \nu(h) \preccurlyeq t^{-\alpha} \quad \mbox{ and }\quad t^{-2}\sum_{|h|_\Sigma \le t} |h|_\Sigma^2 \nu(h) \preccurlyeq t^{-\alpha}.$$
   
    \item [($\mathcal{PP}_{\Sigma, \alpha}$)] 
    For all $h \in \langle \Sigma \rangle$ and finitely supported $f$, 
    $$\sum_{g \in G} |f(gh) - f(g)|^2 \le C|h|_\Sigma^{\alpha} \mathcal{E}_\mu(f,f)$$
\end{description}
 
\begin{theo}
\label{thm: gamma}
    Let $G$ be a finitely generated group of polynomial growth and suppose $\mu = \sum_{i=1}^k \nu_i$
    is a probability measure that is a finite combination of symmetric probability measures.
    Fix a collection of finite subsets $\Sigma_i\subset G$ and $\alpha_i\in (0,2)$, $i\in \{1,\dots,,k\}$. There exists 
    $$\gamma := \gamma(G, \Sigma_1,\ldots,\Sigma_k,\alpha_1,\ldots, \alpha_k) \in \mathbb{R}_+$$ such that:
\begin{enumerate}
    \item 
    if $\nu_i$ satisfies ($\mathcal{U}_{\Sigma_i, \alpha_i}$), $1\le i\le k$, then  $$\mu^{(2n)}(e) \succcurlyeq n^{-\gamma}$$
    \item if $\nu_i$ satisfies ($\mathcal{P}\mathcal{P}_{\Sigma_i, \alpha_i}$),$1\le i\le k$, and the union  $\Sigma_1\cup \ldots \Sigma_k$ generates $G$, then
$$\mu^{(2n)}(e) \preccurlyeq n^{-\gamma}$$

\end{enumerate}
    
\end{theo}

\begin{rem} The number $\gamma$ is described explicitly in Equation \ref{eq:gamma}. 
 \end{rem}

In addition to proving the main theorem above,  we will discuss  the coordinate-wise $\alpha$-stable-like measures of type (\ref{nupsi}) that do
not necessarily satisfy the hypothesis in Theorem \ref{thm: gamma} and prove matching upper and lower bounds for probability of returns using a comparison of Dirichlet forms in Section \ref{sec: polygrowth}. 

\subsection{Weight systems}
The study of the convolution powers 
of measures described in Theorem \ref{thm: gamma} is based on the introduction of an appropriate geometry on the group $G$. This is based on \cite{SCZ-nil,CKSCWZ}. When $G$ itself is nilpotent, the arguments
are detailed in \cite{SCZ-nil}. We first explain the construction in this simpler case.

\subsubsection*{The nilpotent case}
\label{nilcase} Assume $G$ is nilpotent and $\mu$ satisfies the hypotheses in Theorem \ref{thm: gamma} (1) and (2). 
 Note that the group  $\langle \Sigma_i\rangle$ is nilpotent with some growth degree $d_{\Sigma_i}$ and, the hypotheses on $\nu_i$ implies that it satisfies $\nu_i^{(n)}(e)\asymp n^{-d_{\Sigma_i}/\alpha_i}$. 

By hypothesis, $\Sigma := \bigcup_{i=1}^k \Sigma_i$ generates $G$. 
Here, each $\Sigma_i$ is regarded as an alphabet and $\Sigma$ is simply the concatenation of these alphabets where repetition is allowed. Each element $\sigma$ of $\Sigma$ is equipped with an associated $\alpha(\sigma)=\alpha_i\in (0,2)$ if $\sigma$ comes from $\Sigma_i$ or $\Sigma_i^{-1}$ and we set $w(\sigma)=1/\alpha(\sigma)$. This is the weight of $\sigma$ and we propagate this weight along the set of formal commutators of the finite alphabet $\Sigma$ as prescribed in \cite{SCZ-nil}. Informally, 
if $c=[c_1,c_2]$ is a formal commutator of two previously considered formal commutators, then $w(c)=w(c_1)+(c_2)$ (see \cite[Section 1.5]{SCZ-nil}).  Now, for any element $g\in G$, we let
$$
    \|g\|=  \inf\left\{\max_{\xi \in \Sigma} 
    \left\{\left(\mbox{deg}_\xi(\theta) \right)^\frac{1}{w(\xi)} \right\}\;:\; \theta \in \cup_{m=0}^\infty \Sigma
    ^m, g = \theta \in G\right\}
$$

In words, we write $g$ as a finite word $\theta$ over $\Sigma\cup\Sigma^{-1}$. For each $\xi\in \Sigma$, $\mbox{deg}_\xi(\theta)$
is the number of times $\xi$ is used in the word $\theta$. The ``norm'' of the word $\theta$ is the maximum over $\xi\in \Sigma$ of  $\mbox{deg}_\xi(\theta)^{1/w(\xi)}$. And we then minimize over the possible different ways to write $g$ as a word.  Using the key results of \cite{SCZ-nil}, especially \cite[Theorem 2.10]{SCZ-nil}, it follows from the hypotheses $(\mathcal{PP}_{\Sigma_i, \alpha_i})$ for all $\nu_i$ that this measure satisfies
\begin{equation} \label{PPnil} \sum_{g\in G}|f(gh)-f(g)|^2\le C_\mu\|h\|\mathcal E_{\mu}(f,f),  \;\;h\in G,\;f\in L^2(G).
\end{equation}
Moreover, \cite{SCZ-nil} provides a sharp estimate on
$\#\{g\in G: \|g\|\le R\}.$
Namely, there is a positive real $\gamma_\mu$ which is computed  explicitly using the nilpotent structure of $G$ and the weight system $w$ such that
$$\#\{g\in G: \|g\|\le R\}\asymp R^{\gamma_{\mu}}.$$
See \cite[Definition 1.7]{SCZ-nil} and set $\gamma_\mu=D(\Sigma,\mathfrak w)$
where $\mathfrak w$ denote the weight system associated with the weight $w$ defined above (see also Theorem \ref{thm:vol} below). 

 Together with (\ref{PPnil}), this volume estimate gives $\mu^{(2n)}(e)\preccurlyeq n^{-\gamma_\mu}$. A 
 matching lower bound can be derived from assumption $(\mathcal{U}_{\Sigma_i, \alpha_i})$ for each $\nu_i$, which  will be apparent below.

\subsubsection*{Groups of polynomial growth}
If applied directly on a group $G$  of polynomial volume growth, the construction described above for nilpotent groups may fail to capture the key properties of the measure $\mu$. We now follow \cite{CKSCWZ} to construct a proper norm  $\|\cdot\|_G$ in this case. 
 Gromov's theorem states that $G$ has a normal nilpotent subgroup with finite index. For the rest of this section, we assume that $G$ has a normal nilpotent subgroup $N$ with a finite set of right coset representatives denoted $X = \{x_0 = e, x_1, \ldots, x_p\}$ that $\mu$ is a symmetric probability measure satisfying the hypotheses in Theorem \ref{thm: gamma} (1) and  (2). 
\begin{defin} 
\label{Gnorm}
(Norms on $G$)
    Let $S_0$ be a finite  symmetric  generating set of $G$ and for $i=1,\ldots, k$, let $S_i$ be a symmetric finite  generating set of $N_i := N \cap \langle \Sigma_i \rangle$. Let $\Sigma_G$
     be the tuple obtained as the formal union of $S_0, \ldots, S_k$ where repetition is allowed. For $\xi \in \Sigma_G$, set
$$  w_G(\xi) = \begin{cases}
  \frac{1}{\alpha_i} &\text{ if }\xi \in S_i\\
  \frac{1}{2} &\text{ if }\xi\in S_0  
\end{cases}$$
Define the (quasi-)norm $\|\cdot\|_G$ on $G$
     as follows:
     $$\|g\|_G = \inf\left\{\max_{\xi \in \Sigma_G} 
    \left\{\left(\mbox{deg}_\xi(\theta) \right)^\frac{1}{w_G(\xi)} \right\}\;:\; \theta \in \cup_{m=0}^\infty \Sigma_G^m
, g = \theta \in G\right\}$$
     \end{defin}

\begin{defin}\label{Nnorm} (Norms on $N$)
Let $\Xi_0$ be a finite symmetric generating set of $N$ and $i = 1,\ldots, k$, define 
$$\Xi_i = \bigcup_{j=0}^p x_j S_i x_j^{-1}$$
Let $\Sigma_N$ be the tuple which is the formal union of $\Xi_0, \ldots, \Xi_k$ where repetition is allowed. For each $\xi \in \Sigma_N$, set 
$$  w_N(\xi) = \begin{cases}
  \frac{1}{\alpha_i} &\text{ if }\xi \in \Xi_i \\
  \frac{1}{2} &\text{ if }\xi\in \Xi_0  
\end{cases}$$
Define the (quasi-)norm $\|\cdot \|_N$
on $N$ as follows 
 $$\|n\|_N = \inf\left\{\max_{\xi \in \Sigma_N} 
    \left\{\left(\mbox{deg}_\xi(\theta) \right)^\frac{1}{w_N(\xi)} \right\}\;:\; \theta \in \cup_0^\infty (\Sigma_N \cup \Sigma_N^{-1})^m, n = \theta \in N\right\}$$
\end{defin}

The following lemma offers a comparison of volume estimates across different norms.
\begin{lem}[{\cite[Theorem 3.2.1]{CKSCWZ}}]
\label{lem:norm_comp}  For any $g\in N$, $\|g\|_N\asymp \|g\|_G$ and
for any $R \ge 1$, 
    $$\# \{g \in G: \|g\|_G\le R\} \asymp \#  \{g \in N: \|g\|_G\le R\}\asymp \#  \{g \in N: \|g\|_N\le R\} $$
\end{lem}
\begin{rem} The functions $\|\cdot\|_N,\|\cdot\|_G$ are quasi-norms (the triangle inequality holds with a multiplicative constant $C$ which may be different from $C=1$ and depends on the weights). The notation $\|\cdot\|_N,\|\cdot\|_G$, are abuses of notation in so far as the weight system used on $N$ and $G$ to define these quantities actually depend greatly on the measure $\mu$ since the sets $\Sigma_i$ and the positive coefficients $\alpha_i\in (0,2)$ that come from the hypothesis on each component $\nu_i$ of $\mu$
play a key role in the definition of $\|\cdot\|_N,\|\cdot\|_G$. 
\end{rem}

To obtain 
a sharp volume estimate for $\{g \in N: \|g\|_N\le R\}$, apply \cite[Theorem 3.2]{SCZ-nil}. For the convenience of the reader, we describe this key result in the present setting.  
Let $\mathcal{C}(\Sigma_N)$ be the collection of all formal commutators over $\Sigma_N \cup \Sigma_N^{-1}$. 
The weight system $w_N$ can be extended to $\mathcal{C}(\Sigma_N)$ by setting $w_N(\xi^{-1}) =w_N(\xi)$ and $w_N([\xi, \xi']) = w_N(\xi) + w_N(\xi')$. 
For $t \ge 0 $, define
$$N_t := \langle \pi(c) \;|\; c  \in \mathcal{C}(\Sigma_N), w_N(c)  \ge t\rangle \subseteq N $$
where $\pi$ is the evaluation of $c$ in $N$.  Observe that there is a greatest $t$ such that $N_t = N$, call it $w_1$. By induction on the quotients and nilpotency of $N$, there exists a finite sequence of numbers 
$w_1 < w_2 < \ldots < w_j <w_{j+1}$
such that 
$$N_{w_i} \subsetneq N_{w_{i-1}},\quad  N_t = N_{w_i} \text{ for } t \in [w_i, w_{i+1}), \quad N_t = \{e\} \text{ for }t > w_j $$
By construction, $[N, N_{w_i}] \subseteq N_{w_{i+1}}$. Let $A_i$ be the abelian group 
$$A_i := N_{w_i}/N_{w_{i+1}}$$ with torsion free rank $r_i$. Define
\begin{align}
    \label{eq:gamma}
    \gamma_\mu = \sum_{i=1}^j w_i r_i
\end{align}

\begin{theo}[{\cite[Th. 3.2 and Rm. 3.3]{SCZ-nil}}]
\label{thm:vol}
    Referring to the setting and notations above, for any $R >1$, 
    $\# \{g \in N\;:\; \|g\|_N \le R\} \asymp R^{\gamma_\mu}$. 
\end{theo}

\subsection{Proof of Theorem \ref{thm: gamma}}

The upper-bound $\mu^{(2n)}(e)\preccurlyeq n^{-\gamma_\mu}$ follows from the volume estimate of Theorem \ref{thm:vol}, Lemma \ref{lem:norm_comp}, and the pseudo-Poincar\'e inequality (\ref{PPpoly}) below. 
\begin{theo} 
\label{thm: PPpoly}
Take $\mu = \sum_{i=1}^k \nu_i$ where $\nu_i$ satisfies $(\mathcal{PP}_{\Sigma_i, \alpha_i})$ and $\bigcup_{i=1}^k \Sigma_i$ generates $G$. 
Let $\|\cdot\|_G$ be defined as in 
Definition \ref{Gnorm}. Then there exists a constant $C = C(G, \mu)$ such that for any finitely supported function $f$ and $h \in G$,  
\begin{equation}
\label{PPpoly}
  \sum_{g\in G}|f(gh)-f(g)|^2\le C \|h\|_G \mathcal E_\mu(f,f)
\end{equation}
\end{theo}
\begin{proof}

By \cite[Corollary 3.23]{CKSWZ1}, for $h \in G$ with $\|h\|_G = R$, we can write 
$$h = \kappa_0 \prod_{i=1}^q \xi_i^{x_i}\kappa_j$$
where $\xi_i \in \Sigma_G\backslash S_0$, $x_i \in \mathbb{Z}$ with $|x_i| \le R^{w_G(\xi_i)}$ for $1\le i \le q$ and $|\kappa_i|_{S_0}^2\preccurlyeq   R$ for $0\le i \le q$.

Consider the factor $\xi_1^{x_1}$. As $\xi_1 \in \Sigma_G\backslash S_0$, we can find
$j \in \{1,\ldots, k\}$ such that  $\xi_1 \in S_j$ and 
$ w_G(\xi_1) = \frac{1}{\alpha_j}$. Note that $\xi_1 \in S_j \subseteq  \langle \Sigma_j\rangle$ and by the hypothesis 
$(\mathcal{PP}_{\Sigma_j, \alpha_j})$, 
$$\sum_{g \in G}|f(g\xi_1^{x_1}) - f(g)|^2 \le |\xi_1^{x_1}|_{\Sigma_j}^{\alpha_j} \mathcal{E}_{\nu_j}(f,f) \preccurlyeq |x_1|^{\alpha_j} \mathcal{E}_{\nu_j}(f,f) \le R  \mathcal{E}_{\mu}(f,f).$$
In the last step, we also use the fact that $\mathcal{E}_{\nu_j}(f,f)  \le  \mathcal{E}_{\mu}(f,f)$. Such an inequality holds for every factor $\xi_i^{x_i}$ for $i=1,\ldots, q$.

By hypothesis, $\bigcup_{i=1}^k\Sigma_i$ generates $G$ and in particular $S_0$, so each $s\in S_0$ can be written as a product $h_{1}^{(s)} \ldots h_{i_s}^{(s)}$ of elements in $\bigcup_{i=1}^k \Sigma_i$.  Let $\eta_s$ be the minimum of $\mu(h_{i}^{(s)})$ over $i=1,\ldots, i_s$. By Lemma \ref{lem: decomp}, 
$$\sum_{g\in G}|f(gs) - f(g)|^2 \le 
\frac{i_s^2 \mathcal{E}_\mu(f,f)}{ \eta_s}.$$
Now consider the factor $\kappa_i$ where $i=0,1,\ldots, q$. It can be  decomposed over $S_0$ as  
$\kappa_i = s_{1}\ldots s_{n}$ where $n \preccurlyeq   R^{1/2}$. Set $\eta$ be the minimum of $\eta_s$ over all $s \in S_0$ and deduce
\begin{align*}
   \sum_{g\in G}|f(g \kappa_i) - f(g)|^2 \preccurlyeq  n \sum_{m=1}^n \sum_{g\in G} & |f(g s_{m}) - f(g)|^2 \preccurlyeq  \frac{n^2\mathcal{E}_{\mu}(f,f) }{\eta}\preccurlyeq   R\mathcal{E}_{\mu}(f,f)
\end{align*}
The desired result again follows from Lemma \ref{lem: decomp}.

\end{proof}

For the lower bound $\mu^{(2n)}(e) \succcurlyeq n^{-\gamma_\mu}$, the main 
argument is based on the notion of spectral profile. Given a measure $\mu$ on a group $G$, the spectral profile $\Lambda_{2,\mu}$ is the function defined over $[1,\infty)$ by 
$$\Lambda_{2,\mu}(v) = \min\left\{\frac{\mathcal{E}_{\mu}(f,f)}{\|f\|_2^2} \;:\; 1 \le \#\mbox{supp}(f) \le v \right\}$$
Here, 
$\|f\|_2^2 =\sum_{x\in G} |f(x)|^2$.
It's well-known that for any $\gamma > 0$, 
\begin{align}
   & \forall v \ge 1, \Lambda_{2,\mu}(v) \preceq v^{-1/\gamma} \Longleftrightarrow \forall n = 1,2,\ldots, \nu^{(2n)} \succeq n^{-\gamma} \label{eq:lowerbound}\\
   & \forall v \ge 1, \Lambda_{2,\mu}(v)  \succeq v^{-1/\gamma}\Longleftrightarrow \forall n = 1,2,\ldots, \nu^{(2n)} \preceq  n^{-\gamma}
\end{align}
See \cite{SCZlow} and the references given therein.

\begin{theo}
\label{thm:zeta}
Suppose $\mu = \sum_{i=1}^k \nu_i$ where each $\nu_i$ satisfies ($\mathcal{U}_{\Sigma_i, \alpha_i}$) for some  finite set $\Sigma_i$ and $\alpha_i \in (0,2)$. Referring to the setting and notations in Definition \ref{Gnorm}, take $w^* = \max \{w_G(\sigma):\sigma \in \Sigma_G\}$. For every component $\nu_i$ of $\mu$, the function
    $g \mapsto \zeta_R(g) = (R - \|g\|_G^{w^*})_+$ satisfies
    $$\frac{\mathcal{E}_{\nu_i}(\zeta_R, \zeta_R)}{\|\zeta_R\|_2^2} \preccurlyeq   R^{-\frac{1}{w^*}} $$
    Furthermore, for any $n\ge 1$,  $\mu^{(2n)}(e)\succcurlyeq  n^{-\gamma}$. 
\end{theo}

\begin{proof}

Let $W(R) =\# \{g: \|g\|_G\le R^\frac{1}{w^*}\}$ be the volume of the support of $\zeta_R$. 
 By Lemma \ref{lem:norm_comp} and Theorem \ref{thm:vol}, $W(R)\asymp R^\frac{\gamma}{w^*}$.  
If 
$\frac{\mathcal{E}_{\nu_i}(\zeta_R, \zeta_R)}{\|\zeta_R\|_2^2} \preccurlyeq R^{-\frac{1}{w^*}}$ for any component $\nu_i$ of $\mu$, 
$$\Lambda_{2,\mu}(R^\frac{\gamma}{w^*})
\le \sum_{i=1}^k \frac{\mathcal{E}_{\nu_i}(\zeta_R, \zeta_R)}{\|\zeta_R\|_2^2}  
\preceq R^{-\frac{1}{w^*}} = (R^\frac{\gamma}{w^*})^{-\frac{1}{\gamma}} $$
and the lower bound follows from \ref{eq:lowerbound}. 

It remains to show the first statement. Consider the component $\nu_i$ of $\mu$ where $i = 1,\ldots, k$. Since $\zeta_R$ is at least $R/2$ over $\{g: \|g\|_G \le R^\frac{1}{w^*}/2\}$, $\|\zeta_R\|_2^2 \asymp R^2W(R)$. Let $$\Omega  = \{(g,h) \in G\times \langle \Sigma_i\rangle: \zeta_R(gh) + \zeta_R(g) > 0\}$$
By hypothesis $supp(\nu_i) \subseteq \langle \Sigma_i \rangle$ and hence
$$\mathcal{E}_{\nu_i}(\zeta_R,\zeta_R) = \sum_{(g,h) \in \Omega} |f(gh) - f(g)| \nu_i(h)$$
For every fixed $h$, 
    $\# \{g: (g,h)\in \Omega\} \le 2W(R)$, so Condition ($\mathcal{U}_{\Sigma_i, \alpha_i}$) gives 
    \begin{align*}
        \sum_{\substack{ (g,h) \in \Omega\\
        |h|_{\Sigma_i}\ge R^\frac{1}{\alpha_i w^*}}} |\zeta_R(gh) - \zeta_R(g)|^2 \nu_i(h) \le 2R^2W(R) \sum_{|h|_{\Sigma_i} \ge R^\frac{1}{\alpha_i w^*}} \nu_i(h)  \preccurlyeq   R^{2- \frac{1}{w^*}}W(R)
    \end{align*}
We now consider the above sum over $(g,h) \in \Omega$ and $|h|_{\Sigma_i} \le R^\frac{1}{\alpha_i w^*}$. For $(g,h) \in \Omega$, 
$$|\zeta_R(gh) - \zeta_R(g)| \le |\|gh\|_G^{w^*} - \|g\|_G^{w^*} |$$
    Indeed, if $\zeta_R(gh) > 0$ and $  \zeta_R(g) >  0$, the two expressions above are equal;
if only $\zeta_R(gh)$ is positive, $\|g\| \ge R^\frac{1}{w^*}$ and 
$|\zeta_R(gh) - \zeta_R(g)|  = R- \|gh\|_G^{w^*} \le |\|gh\|_G^{w^*} - \|g\|_G^{w^*} | $; the same argument applies if only $\zeta_R(g) $ is positive. 

Let $N_i = \langle S_i \rangle = N \cap \langle \Sigma_i\rangle$ be defined as in  Definition \ref{Gnorm}; it's a normal subgroup of $\langle \Sigma_i\rangle$ with a finite set of right representatives, say $X_i$. Write $h \in \langle \Sigma_i\rangle$ as $h = nx $ where $n \in N_i$ and $x \in X_i$ and  It follows $$|n|_{S_i} \asymp |n|_{\Sigma_i} \asymp |h|_{\Sigma_i}  \le  R^\frac{1}{\alpha_i w^*}. $$ 
Take a word $\theta_n \in \cup_{j=0}^\infty (S_i\cup S_i^{-1})^j$ such that $\theta_n = n$ in $G$ and $ \mbox{deg}_s(\theta_n)  \preccurlyeq R^\frac{1}{\alpha_i w^*}$ for all $s \in S_i$. In particular, this means
$$\|n\|_G \le \max_{s\in S_i} \left\{\mbox{deg}_s(\theta_n)^\frac{1}{w_G(s)} \right\}   = \max_{s\in S_i} \left\{\mbox{deg}_s(\theta_n)^{\alpha_i} \right\}   \preccurlyeq  R^\frac{1}{w^*} $$

Let $M$ be the integer such that  
every element in the representative set $X_i$ can be decomposed as a word over $\Sigma_G$ of length no more than $M$. There exists a word $\theta_x$ over $\Sigma_G$ such that $\theta_x = x$ in $G$ and  
$\mbox{deg}_\sigma(\theta_x) \le M$ for any $\sigma \in \Sigma_G$, as a result of which $\|x\|_G \le M^\frac{1}{2}$. 

It follows for $(g,h) \in \Omega$ and $|h|_{\Sigma_i} \le R^\frac{1}{\alpha_i w^* }$, we must have 
$\|g\|_G \preccurlyeq   R^\frac{1}{w^*} $ and 
$ \|gh\|_G \preccurlyeq   R^\frac{1}{w^*}$. Indeed, since $(g,h) \in \Omega$, either $\|g\|_G \le  R^\frac{1}{w^*}$ or $\|gh\|_G \le  R^\frac{1}{w^*}$; if $\|g\|_G \le R^\frac{1}{w^*}$, 
$$\|gh\|_G\preccurlyeq\|g\|_G + \|n\|_G + \|x\|_G \preccurlyeq  R^\frac{1}{w^*} $$
and if  $\|gh\|_G \le R^\frac{1}{w^*}$,
one can prove $\|g\|_G \le R^\frac{1}{w^*}$ using a similar argument. Take
$\theta_g\in \cup_{j=0}^\infty (\Sigma_G \cup \Sigma_G^{-1})^j$ with $$\|g\|_G = \max_{\sigma \in \Sigma_G} 
    \left\{\left(\mbox{deg}_\sigma(\theta_g) \right)^\frac{1}{w_G(\sigma)} \right\}$$
    and observe
$\mbox{deg}_\sigma(\theta_g)  \preccurlyeq  R^\frac{w_G(\sigma)}{w^*}$ for all $ \sigma \in \Sigma_G$. Combining all the inequalities above, we deduce for 
$(g,h) \in \Omega$ and $|h|_{\Sigma} \le R^{\frac{1}{\alpha_i w^*}}$, if $\|gh\|_G \ge \|g\|_G$. 
\begin{align*}
    & \|gh\|_G^{w^*} - \|g\|_G^{w^*} \\  
    \le &  \max_{\sigma \in\Sigma_G}\left\{ \left(\mbox{deg}_\sigma(\theta_g) +  \mbox{deg}_\sigma(\theta_n) + \mbox{deg}_\sigma(\theta_x)  \right)^\frac{w^*}{w_G(\sigma)}  - \left(\mbox{deg}_\sigma(\theta_g)\right)^\frac{w^*}{w_G(\sigma)}
    \right\} \\  
    \le & \max_{\sigma \in\Sigma_G} \left\{ \left(\mbox{deg}_\sigma(\theta_g) +   \mbox{deg}_\sigma(\theta_n) + \mbox{deg}_\sigma(\theta_x)  \right)^{\frac{w^*}{w_G(\sigma)} -1}  (    \mbox{deg}_\sigma(\theta_n) + \mbox{deg}_\sigma(\theta_x)  )
    \right\} \\
         =&  C_M
\max_{
\substack{s\in S_i\\ \sigma \in \Sigma_G \backslash S_i}}\left\{ \left(\mbox{deg}_\sigma(\theta_g) \right)^{\frac{w^*}{w_G(\sigma)} -1}, \left(\mbox{deg}_s(\theta_g) + 
    \mbox{deg}_s(\theta_n) \right)^{\frac{w^*}{w_G(s)} -1}      \mbox{deg}_s(\theta_n) 
\right\}\\
         \preccurlyeq &
\max_{
\substack{s\in S_i\\ \sigma \in \Sigma_G \backslash S_i}}\left\{ 
R^{1 - \frac{w_G(\sigma)}{w^*}},  R^{1 - \frac{w_G(s)}{w^*}} |n|_{\Sigma_i}
\right\} 
 \preccurlyeq   \max \left\{ 
R^{1 - \frac{1}{2w^*}},  R^{1 - \frac{1}{\alpha_i w^*}} |n|_{\Sigma_i} 
\right\} := I
\end{align*}
In the last inequality, we use the fact that 
$ \frac{1/2}{w^*}\le \frac{w_G(\sigma)}{w^*} \le 1$ for any $\sigma\in \Sigma_G \backslash S_i$.  
If $\|gh\|_G < \|g\|_G$, we can run the same argument with $g' = gh$ and $h' = h^{-1}$ and obtain the same bound up to a constant.  If $ I  = R^{1 - \frac{1}{2w^*}} $, by 
\begin{align*}
      \sum_{
     \substack{ (g,h) \in \Omega\\
     |h|_{\Sigma_i}\le R^\frac{1}{\alpha_i w^*} } } |\zeta_R(gh) - \zeta_R(g)|^2 \nu_i(h) & \le  2W(R) R^{2 - \frac{1}{w^*}}\sum_{|h|_{\Sigma_i} 
      \le R^\frac{1}{\alpha_i w^*}} \nu_i(h)\\ & \le  2W(R) R^{2 - \frac{1}{w^*}}
\end{align*}
If $ I  =  R^{1 - \frac{1}{\alpha_i w^*}} |n|_{\Sigma_i} $, by the hypothesis ($\mathcal{U}_{\Sigma_i, \alpha_i}$) 
\begin{align*}
     \sum_{
     \substack{ (g,h) \in \Omega\\
     |h|_{\Sigma_i}\le R^\frac{1}{\alpha_i w^*} } } |\zeta_R(gh) - \zeta_R(g)|^2 \nu_i(h)
    & \le   2W(R)R^{2 - \frac{2}{\alpha_i w^*}}\sum_{|h|_{\Sigma_i} \le R^\frac{1}{\alpha_i w^*}}|n|_{\Sigma_i}^2 \nu_i(h) \\
    & \preccurlyeq  2W(R)R^{2 - \frac{2}{\alpha_i w^*}}\sum_{|h|_{\Sigma_i} \le R^\frac{1}{\alpha_i w^*}}|h|_{\Sigma_i}^2 \nu_i(h) \\
     & \preccurlyeq  W(R)R^{2 - \frac{2}{\alpha_i w^*}} R^\frac{2-\alpha_i}{\alpha_i w^*} = W(R)R^{2- \frac{1}{w^*}}
\end{align*}
completing the proof of the first statement.

\end{proof}

    Having proved the pseudo-Poincar\'e inequality (\ref{PPpoly}) and the two-sided bound, $\mu^{(2n)}(e)\asymp n^{-\gamma_\mu}$ of Theorem \ref{thm: gamma} with $\gamma_\mu$ related to the volume of $\{g: \|g\|_G\le R\}$ by Theorem \ref{thm:vol} and Lemma \ref{lem:norm_comp}, we 
can follow the argument in \cite[Theorem 5.5]{CKSWZ1} and obtain the following analog of Theorem \ref{thm:1b}. 

\begin{theo}  Let $G$ be a finitely generated group with $V(r) \asymp r^d$, $\mu$ be a symmetric probability measure satisfying the hypotheses (1) and (2) of Theorem \ref{thm: gamma}, and 
$\gamma$ be as in Theorem \ref{thm: gamma}. 
Denote by $\|\cdot\|_G$ the induced quasi-norm defined as in Definition \ref{Gnorm}.  
 \begin{itemize}
     \item There exists a constant $C$ such that, for all integers $n,m$ and group elements $x,y\in G$,
     \begin{equation*}|\mu^{(n+m)}(xy)-\mu^{(n)}(x)|\le C \left(\frac{m}{n}+ \sqrt{\frac{\|y\|_G}{n}}\right)\mu^{(2\lceil n/2\rceil)}(e).\end{equation*}
     \item There exists $\eta$ and $n_0$ such that for all $n\ge n_0$ and $x\in G$ such that $\|g\|_G\le \eta n$, it holds that
    \begin{equation*}
        \mu^{(n)}(g)\asymp n^{-\gamma}
    \end{equation*}
    \item Let $(X_n)_{n=0}^\infty$ be random walk on $G$ driven by $\mu$. 
    For all $\epsilon>0$ there exists $\beta>0$ such that
    for all $n$,
    \begin{equation*}
        \mathbf P_e\left(\sup_{k\le n}\{\|X_k\|_G\}\ge \beta n\right)
\le \epsilon.    \end{equation*}

  \item
There exist $\epsilon, \beta_1 \in (0,\infty)$ and $\beta_2 \ge 1$ such that for all $n, \tau$ with $\frac{1}{2} \tau /\beta_1 \le n \le \tau/\beta_1$, 
$$\inf_{x: \|x\|_G \le \tau } P_x\left( \sup_{0\le k \le n} \{\|X_k\|_G\} \le \gamma_2 \tau, \|X_n\|_G \le \tau \right)\ge \epsilon.$$
 \end{itemize}
 \end{theo}

\subsection{Coordinate-wise stable-like measures revisited}

\label{sec: polygrowth}
Let $G$ be a finitely generated group with $V(r) \asymp r^d$. 
Fix a $k$-tuple $S = (s_1,\ldots, s_k)$ of elements in $G$ such that $\langle S \rangle = G$ and a map $\psi: \mathbb{Z}^k \to \mathbb{R}$ given by $\psi(\bar{a}) = (1+ \|a\|)^{-k-\alpha}$. In this section, we'll revisit the coordinate-wise $\alpha$-stable-like measure $\nu_{\psi, S}$ given by \ref{nupsi}: 
\begin{align*}
 \nu_{\psi, S}(h)=  
    \begin{cases}
     \frac{1}{2}\sum\limits_{ \bar{a}: 
      \pi(\bar{a})\in \{h,h^{-1}\} } \psi(\bar{a}) &\text{if }\{h,h^{-1}\}\cap \pi(\mathbb Z^k)\neq \emptyset \\
         0 &\text{otherwise }\\
\end{cases}
\end{align*}
where 
$$\pi:\mathbb Z^k\to G, \bar{a}=(a_1,\dots,a_k)\mapsto \pi(\bar{a})=s_1^{a_1}\dots s_k^{a_k}$$
and derive results analogous 
to those in Theorem \ref{thm:2}. As discussed in Remark \ref{rem: polypp}, since the support of $\nu_{\psi, S}$ is no longer restricted to a nilpotent group now, the techniques employed in the proof of Theorem \ref{thm:2} are no longer applicable. In fact, in this case, the entire argument relies on the following theorem, the proof of which is postponed until the end of the section.
\begin{theo}
\label{thm: diri-comp}
Define a measure $\mu_{S,\alpha}$ by 
\begin{align}
    \label{eq: mu_S_alpha}
    \mu_{S,\alpha}(g)\asymp \sum_{i=1}^k \sum_{a \in \mathbb{Z}} \frac{\textbf{1}_{s_i^a} (g) }{(1+|a|)^{1+\alpha}}
\end{align}
    The Dirichlet forms of  $\mu_{S,\alpha}$ and $ \nu_{\psi, S}$ are comparable. 
\end{theo}

The measure $\mu_{S,\alpha}$ is straightforward to handle as it satisfies the hypothesis in Theorem \ref{thm: gamma}
trivially with $\Sigma_i = \langle s_i\rangle$ and $\alpha_i = \alpha$ for $i=1,\ldots, k$. 
The following corollary is a easy consequence of Theorem \ref{thm: gamma}, Lemma \ref{lem:norm_comp}, 
Theorem \ref{thm:vol}, Theorem \ref{thm: PPpoly}, and 
Theorem \ref{thm:zeta}.
\begin{cor}
    \label{cor: poly_mu_S_alpha}
    Let $\|\cdot\|_G$ be the quasi-norm on $G$ induced by $\mu_{S,\alpha}$ by choosing $\Sigma_i$, $i=1,\ldots, k$, in Definition \ref{Gnorm} to be $\langle s_i\rangle$. There exists $$\gamma := \gamma(G, \langle s_1\rangle, \ldots \langle s_k\rangle, \alpha, \ldots,\alpha)$$ given as in Theorem \ref{thm: gamma} such that 
$\# \{g\in G: \|g\|_G \le R\} \asymp R^\gamma$ and 
           $$\sum_{g\in G}|f(gh) - f(g)|^2 \le C\|h\|_G \,\mathcal{E}_{\mu_{S,\alpha}}(f,f) $$
For every $v \ge 1$,  $\Lambda_{2,\mu_{S,\alpha}}(v) \preccurlyeq v^{-1/\gamma}$. Furthermore, $\mu_{S,\alpha}^{(2n)}(e) \asymp n^{-\gamma}$. 
\end{cor}

The following corollary following directly from Theorem \ref{thm: diri-comp} and Corollary  \ref{cor: poly_mu_S_alpha} shows the same statements are also true for $\mu$. 
\begin{cor}
 \label{cor: poly_mu_psi_S}
Let $\gamma$ be defined as in Corollary  \ref{cor: poly_mu_S_alpha} and $\|\cdot\|_G$ be the quasi-norm induced by $\mu_{S,\alpha}$. Then the pseudo-Poincar\'e of $\nu_{\psi, S}$ is controlled by $\|\cdot\|_G$, i.e. 
$$\sum_{g\in G}|f(gh) - f(g)|^2 \le C\|h\|_G \,\mathcal{E}_{ \nu_{\psi, S}}(f,f) $$
For every $v \ge 1$, $\Lambda_{2,\nu_{\psi, S}}(v)   \preccurlyeq v^{-1/\gamma}$. Furthermore, $ \nu_{\psi, S}^{(2n)}(e) \asymp n^{-\gamma}$. 
\end{cor}

With the following illustrative example, we highlights how the lack of nilpotency affects the behaviors of a coordinate-wise $\alpha$-stable-like measure, more specifically, the exponent $\gamma$ in Corollary  \ref{cor: poly_mu_S_alpha} and Corollary  \ref{cor: poly_mu_psi_S}.        
\begin{exa}
\label{exa: nil}
Consider $G = \mathbf{D}\times \mathbb{Z}$ where $\mathbf{D} = \langle u, v : u^2 = v^2\rangle$ is the infinite Dihedral group. It is not nilpotent because $[(uv)^n, u] = (uv)^{2n}$ and contains $N:= \langle uv \rangle \times \mathbb{Z}$ as the nilpotent subgroup with quotient $\langle u\rangle\simeq \mathbb{Z}_2$. Let 
$$S= \{s_1 = (u,0), s_2 = (v,0), s_3 = (0, 1)\}$$
Let $\nu_{\psi, S}$ be a coordinate-wise measure of type \ref{nupsi} where $\psi: \mathbb{Z}^3 \to \mathbb{R}, \bar{a} \mapsto (1+\|a\|)^{-3-\alpha}$. The associated $\mu_{S,\alpha}$ is
given by 
$$\mu_{S,\alpha}(g) \asymp \sum_{i=1}^3 \sum_{a\in \mathbb{Z}} \frac{\mathbf{1}_{s_i^a}(g)}{(1+|a|)^{1+\alpha}}$$
Construct $\|\cdot\|_G$ induced by $\mu_{S,\alpha}$ and $\gamma_{\mu_{S,\alpha}}$  
as in Definition \ref{Gnorm} and Equation \ref{eq:gamma}. The $\gamma$ in Corollary  \ref{cor: poly_mu_S_alpha} is  exactly $\gamma_{\mu_{S,\alpha}} = \frac{1}{2} + \frac{1}{\alpha}$. 

If $G$ were a nilpotent group, Theorem \ref{thm:2} would give $\gamma = \frac{2}{\alpha} = \frac{1}{\alpha} + \frac{1}{\alpha}$ as the volume function of $G$ is $V(r) = r^2$. In the correct estimate, one of the $\frac{1}{\alpha}$ terms is replaced with $\frac{1}{2}$ precisely because $\langle s_1\rangle$ and $\langle s_2\rangle $ intersect $N$ trivially, causing the generator $s_1s_2$ of $N$ to have a weight of $\frac{1}{2}$ instead of $\frac{1}{\alpha}$. 
\end{exa}

Building on Theorem \ref{thm: diri-comp}, the preceding argument readily generalizes to the case of finite convex combinations of coordinate-wise stable-like measures of the form $\nu_{\psi, S}$. 
Let $\bm{\alpha} = (\alpha_1,\ldots, \alpha_m) \in (0,2)^m$ and 
$\bm{S} = (S_1,\ldots, S_m)$ be a collection of finite-length tuples of elements in $G$. For $i=1,\ldots, m$, define $\psi_i: \mathbb{Z}^{|S_i|} \to \mathbb{R}$ by $\psi_i(\bar{a}) = (1+\|a\|)^{-|S_i| + \alpha_i}$. Define 
$$\nu_{\bm{S}, \bm{\alpha}} = \sum_{i=1}^m \nu_{\psi_i, \alpha_i} $$
Denote by $\Tilde{\bm{S}}$ the formal concatenation $S_1 \sqcup \ldots \sqcup S_m$ and 
$$\Tilde{\bm{\alpha}} = (\underbrace{\alpha_1, \ldots ,\alpha_1}_{|S_1|-times},  \ldots, \underbrace{\alpha_m, \ldots ,\alpha_m}_{|S_m|-times} ) $$
Set $z:= \sum_{j=1}^m |S_j|$. Define $\mu_{\Tilde{\bm{S}} , \Tilde{\bm{\alpha}} }$ by 
$$\mu_{\Tilde{\bm{S}} , \Tilde{\bm{\alpha}} }(g) = \sum_{i=1}^z \sum_{a \in \mathbb{Z}} \frac{\textbf{1}_{\Tilde{\bm{S}} _i^a} (g) }{(1+|a|)^{1+\Tilde{\bm{\alpha}}_i}}$$
where $\Tilde{\bm{S}}_i$ and $\Tilde{\bm{\alpha}}_i$ denote the $i$-th entry in $\Tilde{\bm{S}}$ and $\Tilde{\bm{\alpha}}$ respectively. By Theorem \ref{thm: diri-comp}, we have
$$\mathcal{E}_{\nu_{\bm{S}, \bm{\alpha}}} (f,f) \asymp  \mathcal{E}_{\mu_{\Tilde{\bm{S}} , \Tilde{\bm{\alpha}} }} (f,f) $$
The same argument as in Corollary \ref{cor: poly_mu_S_alpha} and Corollary \ref{cor: poly_mu_psi_S} shows 
 \begin{cor}
    Let $\|\cdot\|_G$ be the quasi-norm on $G$ induced by $\mu_{\Tilde{\bm{S}} , \Tilde{\bm{\alpha}} }$ by choosing $\Sigma_i$, $i=1,\ldots z$, in Definition \ref{Gnorm} to be the subgroup generated by $\Tilde{\bm{S}}_i$. Construct $$\gamma := \gamma(G, \langle \Tilde{\bm{S}}_1 \rangle, \ldots \langle \Tilde{\bm{S}}_z\rangle, \Tilde{\bm{\alpha}}_1, \ldots,\Tilde{\bm{\alpha}}_z)$$ given as in Theorem \ref{thm: gamma} 
    Then the pseudo-Poincar\'e of $\nu_{\bm{S}, \bm{\alpha}}$ is controlled by $\|\cdot\|_G$, i.e. 
$$\sum_{g\in G}|f(gh) - f(g)|^2 \le C\|h\|_G \,\mathcal{E}_{ \nu_{\bm{S}, \bm{\alpha}}}(f,f) $$
For every $v \ge 1$, $\Lambda_{2,\nu_{\bm{S}, \bm{\alpha}} }(v)   \preccurlyeq v^{-1/\gamma}$. Furthermore, $ \nu_{\bm{S}, \bm{\alpha}}^{(2n)}(e) \asymp n^{-\gamma}$. 
\end{cor}

The rest of this subsection is dedicated to proving Theorem \ref{thm: diri-comp}. To this end, we begin with a computational lemma, which can be proved using a simple integral approximation argument. 

\begin{lem}
For $m \le k$ and any fixed $(v_{m+1}, \ldots, v_k) \in \mathbb{Z}^{m-k}$
    \label{lem: sum}
    $$\sum_{(v_1, \ldots, v_m) \in \mathbb{Z}^{m}} \left(1+ \sum_{i=1}^k |v_i| \right)^{-(\alpha +k)} \asymp \left(1 + \sum_{i=m+1}^k |v_i|\right)^{-(\alpha + k -m)} $$
\end{lem}

\begin{proof} (of Theorem \ref{thm: diri-comp}) 
In this proof, for a $k$-dimensional vector $\bar{v}$, $\bar{v}_i$ denotes the $i$-th entry of $\bar{v}$ and $|\bar{v}| := \sum_{i=1}^k |\bar{v}_i|$ denotes the 1-norm. For simplicity, we write $\nu: =\nu_{\psi, S}$ and $\mu:= \mu_{S,\alpha}$. 

 By Lemma \ref{lem: decomp} and Lemma \ref{lem: sum}, 
    \begin{align*}
         &\mathcal{E}_{\nu}(f,f) =   \sum_{\bar{a}\in \mathbb{Z}^k} \frac{|f(g \pi(\bar{a})) - f(g)|^2 }{(1 + |\bar{a}|)^{\alpha + k}} \preccurlyeq   \sum_{j=1}^k  \sum_{\bar{a}\in \mathbb{Z}^k}
        \frac{|f(g s_j^{\bar{a}_j}) - f(g)|^2 }{(1 + |\bar{a}|)^{\alpha + k}} \\
         = & \sum_{j=1}^k  \sum_{
        \substack{ \bar{v} \in \mathbb{Z}^{k-1}\\a \in \mathbb{Z} }
     }
        \frac{|f(g s_j^{a}) - f(g)|^2 }{(1 + |a| + |\bar{v}|)^{\alpha + k}}  \asymp \sum_{j=1}^k  \sum_{
       a \in \mathbb{Z} }
        \frac{|f(g s_j^{a}) - f(g)|^2 }{(1 + |a|)^{\alpha + 1}}  = \mathcal{E}_{\mu}  (f,f)  
    \end{align*}

For the other direction of inequality, we'll prove $\mathcal{E}_{\mu_{\{s_i\},\alpha}} (f,f)  \le \mathcal{E}_{\nu } (f,f)$ by induction on $i$. Consider the base case $\mathcal{E}_{\mu_{\{s_1\},\alpha}}(f,f)  $. 
Take some $a \in \mathbb{Z}^1$ and $\bar{v} \in \{0\} \times \mathbb{Z}^{k-1}$, i.e. the collection of $k$-dimensional integer-valued vector with the first entry 0. We can write
$s_1^{a } = s_1^{2a}  \pi(\bar{v}) \cdot \pi(\bar{v})^{-1} s_1^{-a}$. Again, by Lemma \ref{lem: decomp} and Lemma \ref{lem: sum}, 
\begin{align*}
&\mathcal{E}_{\mu_{\{s_1\},\alpha}} (f,f)  
 =  \sum_{a \in \mathbb{Z}}\sum_{g\in G} |f(gs_1^{a}) - f(g)|^2  \frac{1}{(1+ |a|)^{1+\alpha}} \\
    &\asymp \sum_{a \in \mathbb{Z}}\sum_{g\in G} |f(gs_1^{a}) - f(g)|^2  \sum_{
\bar{v}  \in 0\times \mathbb{Z}^{k-1}} \frac{1}{(1 + |a| + |\bar{v}|)^{k+\alpha}} \\
   & \preccurlyeq 
    \sum_{ 
\substack{ a \in \mathbb{Z}\\\bar{v} \in 0\times \mathbb{Z}^{k-1}}}  
   \frac{\sum_{g\in G} |f(g) - f(gs_1^{2a_1}  \pi(\bar{v}))|^2}{(1 + |a| + |\bar{v}|)^{k+\alpha}} +   \sum_{ 
\substack{ a \in \mathbb{Z}\\\bar{v} \in 0\times \mathbb{Z}^{k-1}}}  
   \frac{\sum_{g\in G} |f(gs_1^{a_1}  \pi(\bar{v})) - f(g)|^2 }{(1 + |a| + |\bar{v}|)^{k+\alpha}}\\
   & \preccurlyeq \mathcal{E}_{\nu}(f,f)
\end{align*}
Now consider $1 < p \le k$. 
For any $a\in \mathbb{Z}$, and $k$-dimensional vectors $\bar{v}$ and $\bar{w}$, we can decompose 
\begin{align*}
    s_p^{a} = 
  \left( \prod_{i=p-1}^{1} s_i^{-\bar{v}_i}\right) \cdot \left(
    \pi(\bar{v}) s_p^{2a}   \pi(\bar{w}) \right) \cdot \left(\pi(\bar{w})^{-1}s_p^{-a}  \pi(\bar{v})^{-1}\right) \cdot \left(\prod_{i=1}^{p-1} s_i^{\bar{v}_i} \right)
\end{align*}

Let $V$ and $W$ denote the subspaces of $\mathbb{Z}^k$ consisting, respectively, of vectors whose last \(k-p+1\) entries and whose first \(p\) entries are zero, i.e. 
$$V:= \mathbb{Z}^{p-1} \times \{0\}^{k-p+1} \text{ and }W:= \{0\}^{p} \times \mathbb{Z}^{k-p}$$
The decomposition above and repetitive application of  Lemma \ref{lem: decomp} and Lemma \ref{lem: sum} give: 
\begin{align*}
 &    \mathcal{E}_{\mu_{\{s_p\},\alpha}}(f,f)  = \sum_{a\in \mathbb{Z}} \sum_{g\in G}   \frac{|f(gs_p^{a}) - f(g)|^2}{(1+ |a|)^{1+\alpha}}\asymp  \sum_{\substack{ a \in \mathbb{Z}\\\bar{v} \in V\\ \bar{w} \in W} } \sum_{g\in G}    
\frac{|f(gs_p^{a}) - f(g)|^2 }{ (1+ |a| + |\bar{v}| + |\bar{w}|)^{k+\alpha}}\\
&\preccurlyeq    \sum_{\substack{ a \in \mathbb{Z}\\\bar{v} \in V\\ \bar{w} \in W} }  \sum_{g\in G}    \frac{ | f(g  \pi(\bar{v}) s_p^{2a}   \pi(\bar{w}))- f(g )|^2 }{(1+ |a| + |\bar{v}| + |\bar{w}|)^{k+\alpha}}  + \sum_{\substack{ a \in \mathbb{Z}\\\bar{v} \in V\\ \bar{w} \in W} }  \sum_{g\in G}    \frac{ |f(g \pi(\bar{v}) s_p^{a}   \pi(\bar{w})) - f(g) |^2 }{(1+ |a| + |\bar{v}| + |\bar{w}|)^{k+\alpha}}\\
 &\quad  \quad +  2 \sum_{i=1}^{p-1} \sum_{\substack{ a \in \mathbb{Z}\\\bar{v} \in V\\ \bar{w} \in W} }  \sum_{g\in G}    \frac{  |f(g) - f(gs_i^{\bar{v}_i})|^2 }{(1+ |a| + |\bar{v}| + |\bar{w}|)^{k+\alpha}} \\
 & 
 \preccurlyeq 
2 \mathcal{E}_{\nu}(f,f)  + 
  2 \sum_{i=1}^{p-1} \sum_{\bar{u}\in \mathbb{Z}^k }  \sum_{g\in G}    \frac{  |f(g) - f(gs_i^{\bar{u}_i})|^2 }{(1+  |\bar{u}|)^{1+\alpha}} \preccurlyeq \mathcal{E}_{\nu}(f,f) + 2\sum_{i=1}^{p-1} \mathcal{E}_{\mu_{\{s_i\},\alpha}}(f,f) 
     \end{align*}
In the second last step, when simplifying the first two sums, we use the fact that any
$ \pi(\bar{v}) s_p^{2a}   \pi(\bar{w})$ or $ \pi(\bar{v}) s_p^{a}   \pi(\bar{w})$ appearing  in the sum is assigned by the measure $\nu$ mass comparable to $(1+ |a| + |\bar{v}| + |\bar{w}|)^{-(k+\alpha)}$. 
Finally, by the induction hypothesis, the last expression is bounded by aconstant multiple of $\mathcal{E}_{\nu}(f,f)$ as desired. 
\end{proof}

\appendix

\section{Mal'cev Basis}
\label{app: Malcev}
 Let $G$ be a finitely generated torsion-free nilpotent group. It has a descending central series 
\begin{align}
\label{series}
    G = G_1 \rhd G_2 \rhd \ldots \rhd G_{n+1} = 1
\end{align}
that is poly-infinite cyclic, i.e. $G_i/G_{i+1}$ is infinite cyclic for $1\le i\le n$. The number of infinite cyclic factors $n$ is an invariant of $G$ called the \textit{torsion-free rank} of $G$.  For each $1\le i\le n$, let $u_i$ be an element satisfying $G_i = gp(G_{i+1}, u_i)$. Then every element of $G$ can be uniquely expressed in the normal form $u_1^{\alpha_1} \ldots u_n^{\alpha_n}$ or
$u_n^{\beta_n} \ldots u_1^{\beta_1}$ where the exponents are integers. The tuple $(u_1,\ldots, u_n)$ is called a Mal'cev basis adpated to the series \ref{series}. 
\begin{exa}
\label{exa:free}
    Let $F_m/\gamma_{c+1}F_m$ be a free nilpotent group of class $c$, generated by the alphabet $X=\{x_1,\ldots,x_m\}$. A basic commutator $b_j$ of length $l$ is defined inductively as follows (see for instance \cite{CMZ})
    \begin{enumerate}
        \item The elements of $X$ are the basic commutators of length 1. We impose an arbitrary ordering on these and relabel them as $b_1,\ldots,b_k$ where $b_i < b_j$ if $i < j$.
        \item Suppose that we have defined and ordered the basic commutators of length less than $l > 1$. The basic commutators of length $l$ are $[b_i,b_j]$ where
        \begin{enumerate}
            \item $b_i$ and $b_j$ are basic commutators and $len(b_i) + len(b_j) = l$,
            \item $b_i > b_j$ and 
            \item if $b_i  = [b_s,b_t]$, $b_j \ge b_t$
        \end{enumerate}
        \item Basic commutators of length $l$ come after all basic commutators of length less than $l$ and are ordered arbitrarily with respect to one another. 
    \end{enumerate}
For instance, if we impose and ordering $x_1<x_2<x_3$ on the generators of $F_3/\gamma_4 F_3$, one of the basic sequences of basic commutators (of weight at most 3) is the following (ordered) sequence
    \begin{align*}
        &x_1,x_2,x_3,[x_2,x_1],[x_3,x_1],[x_3,x_2], \\&[[x_2,x_1],x_1],[[x_2,x_1],x_2],[[x_2,x_1],x_3],[[x_3,x_1],x_1],\\
        & [[x_3,x_1],x_2],[[x_3,x_1],x_3],[[x_3,x_2],x_2],[[x_3,x_2],x_3]. 
    \end{align*}
Any (ordered) sequence of basic commutators of length at most $c$ is a set of Mal'cev adapted to some refinement of the lower central series of $F_k/\gamma_{c+1}F_k$.
\end{exa}

\begin{exa}
\label{exa: matrix_coord}
    Take $G$ to be the dimension-4 unipotent matrix group. Let $M_{ij}$ be the matrix with all entries set to zero except for a 1 in the $(i,j)$ position and along the diagonal. Consider the descending series 
    \begin{align*}
        G     \rhd   &\langle M_{14}, M_{24}, M_{34}, M_{13},M_{23}\rangle \rhd \langle M_{14}, M_{24}, M_{34}, M_{13}\rangle \\
        & \rhd \langle M_{14}, M_{24}, M_{34}, M_{13}\rangle 
         \rhd \langle M_{14}, M_{24}, M_{34}\rangle \rhd \langle M_{14}, M_{24}\rangle \rhd \langle M_{14}\rangle \rhd \{e\}
    \end{align*}
    Simple calculation shows that it's central and is clearly polycyclic. 
The adapted Mal'cev basis 
$(M_{12}, M_{23} ,M_{13} , M_{34} , M_{24} ,M_{14} )$
gives the matrix-coordinate system \begin{align*}
    \pi_S(a_{14}, a_{24}, a_{34}, a_{13},a_{23},a_{12}) & = 
    M_{14}^{a_{14}}M_{24}^{a_{24}} M_{34}^{a_{34}} M_{13}^{a_{13}}M_{23}^{a_{23}}M_{12}^{a_{12}}\\
    &= \begin{pmatrix}
    1 & a_{12} & a_{13} & a_{14}\\
    0 & 1 & a_{23} & a_{24}\\
    0 & 0 & 1& a_{34} \\
    0 & 0 & 0 & 1
\end{pmatrix} 
\end{align*}
\end{exa}

\begin{exa}
A more canonical way to construct the central polycyclic  descending series is via the lower central series. 
Take $G:= G_n$ to be the  $n$-dimensional unipotent matrix group. 
Consider the lower central series of $G$
$$G  = \gamma_1G \ge \gamma_2 G \ge \ldots \ge \gamma_{n-1}G \ge \gamma_{n}G = \{e\}$$
For each $i=1,\ldots, n-1$, $\gamma_k G$ is exactly the subgroup generated by 
$$\{ M_{i,i+k}: i=1,\ldots, n-k\}$$
Here, $M_{i,i+k}$ is the matrix
where all entries are zero except for a 1 at the position $(i, i+k)$ and along the diagonals. Between each pair subgroups $\gamma_k G $ and $\gamma_{k+1}G$,
we construct a refinement where the factor group of any two consecutive terms is infinite-cyclic:   
\begin{align*}
    \gamma_k G  > \gamma_k G/\langle M_{n-k,n}\rangle >  &\gamma_k G/\left\langle \bigcup_{i=n-k-1}^{n-k} M_{i,i+k} \right\rangle 
  \\> & \ldots 
> G/\langle \bigcup_{i=2}^{n-k} M_{i,i+k} \rangle > \gamma_{k+1}G
\end{align*}
The (ordered) set of generators of the factor groups are taken to be 
$$\mathcal{M}_k :=( M_{1,1+k},M_{2, 2+k},\ldots, M_{n-k,n})$$
The concatenation 
\begin{align*}
    \mathcal{M} &:= \mathcal{M}_1 \sqcup \ldots \sqcup \mathcal{M}_{n-1} \\
    & = ( M_{1,2}, M_{2,3}, \ldots, M_{n-1,n}, M_{1,3},\ldots, M_{1,n-1}, M_{2,n},M_{1,n}) 
\end{align*}
preserving the specified ordering of each $\mathcal{M}_\bullet$, is a Mal'cev basis.
In the case $n = 4$, the Mal'cev basis is 
$$(M_{12}, M_{23}, M_{34}, M_{13}, M_{24}, M_{34})$$
It endows $G$ with a coordinate system distinct from the matrix-based coordinates introduced in the earlier example. 

\end{exa}

\begin{exa} Consider the group $G=\langle s_1,\dots, s_n| [s_i,s_j]=1, 1<i<j\le n, [s_1,s_j]=s_{j+1}, 1<j<n-1, [s_1,s_n]=1 \rangle$. This group is generated by $s_1,s_2$ and the tuple $(s_1,\dots,s_n)$
 is a Mal'cev basis for $G$ (associated with the obvious descending series). Any element $g$ of the roup can be written uniquely as $g=s_1^{x_1}\dots s_n^{x_n}$.
\end{exa}

Given any finitely generated torsion free nilpotent group $G$ equipped with a Mal'cev basis $(u_1,\dots,u_n)$. Pick any generating subset $S=\{u_{i_1},\dots,u_{i_k}\}$ containing $k$ elements from this basis and an $\alpha\in (0,2)$. Set \[\pi(\bar{x})=u_{i_1}^{x_1}\dots u_{i_k}^{x_k},\;
\psi(\bar{x})= \frac{c_{k,\alpha}}{(1+\|\bar{x}\|_2^2)^{(k+\alpha)/2}}, \;\;\bar{x}=(x_1,\dots x_k)\in \mathbb Z^k,\]  where $\|\bar{x}\|_2^2=\sum_{i-1}^k|x_i|^2$ and  $c_{k,\alpha}$ is the normalizing constant making $\psi$ a probability distribution on $\mathbb Z^k$.  This and (\ref{nupsi}) describe a collection of coordinate-wise stable like measure on $G$ to which the results of this paper apply.

\end{document}